\documentclass[12pt,reqno]{amsart}
\usepackage[all,cmtip]{xy}
\usepackage{color,graphicx,amssymb,latexsym,amsfonts,txfonts,amsmath,amsthm}
\usepackage{pdfsync}
\usepackage{enumitem}
\usepackage{anyfontsize}
\input epsf
\advance\hoffset by -0.9 truecm

\usepackage{hyperref}
\hypersetup{
    colorlinks=true,       
    linkcolor=blue,          
    citecolor=blue,        
    filecolor=blue,      
    urlcolor=blue           
}

\newcommand{\s}{\vspace{0.15cm}}

\newtheorem{coro}{Corollary} 
\newtheorem{lemm}{Lemma}

\theoremstyle{remark}                  
\newtheorem{rema}{\bf Remark}
\newtheorem{exem}{\bf Example}
\usepackage[normalem]{ulem}

\begin{document}

\title{Algebraic models of cyclic $k$-gonal curves}

\author{Rub\'en A. Hidalgo}
\address{Departamento de Matem\'atica y Estad\'{\i}stica\\
Universidad de La Frontera, Francisco Salazar 01145, Temuco, Chile}
\email{ruben.hidalgo@ufrontera.cl}

\author{Sebasti\'an Reyes-Carocca}
\address{Departamento de Matem\'aticas, Facultad de Ciencias, Universidad de Chile, Las Palmeras, 3425 \~Nu\~noa, Santiago, Chile}
\email{sebastianreyes.c@uchile.cl}

\thanks{Partially supported by ANID Fondecyt Regular  Grants 1230001,  1220099 and 1230708}

\begin{abstract} 
In this paper, we describe explicit algebraic equations of tame cyclic $k$-gonal curves, where $k \geq 2$ is an integer, reflecting the action of the normalizer of a tame cyclic $k$-gonal automorphism. For $k$ a prime integer, this was previously done by A. Wootton.
\end{abstract}

\maketitle
\thispagestyle{empty}

\section{Introduction}
Groups of conformal automorphisms of closed Riemann surfaces is a classical area of research in complex and algebraic geometry, dating back to the nineteenth century, with contributions from Riemann, Klein, and others. Let $S$ be a closed Riemann surface of genus $g \geq 2$, and let ${\rm Aut}(S)$ be its group of conformal automorphisms. In \cite{Schwarz}, Schwarz observed that ${\rm Aut}(S)$ is finite and, in \cite{Hurwitz}, Hurwitz noted that its order is bounded above by $84(g-1)$. In \cite{Greenberg}, Greenberg proved that every finite group can be realized as a group (even the full group) of conformal automorphisms of some closed Riemann surface. As a consequence of the Riemann-Roch theorem, $S$ can be described by an irreducible complex (projective or affine) algebraic curve. Given some subgroup $H \leq {\rm Aut}(S)$, we can ask if we may find one of these algebraic curves from which the action of $H$ can be read. Except for particular cases, this is a computationally challenging task. In this paper, we restrict ourselves to the case of cyclic gonal surfaces.

A {\it $k$-gonal automorphism} of $S$ is an element
$\tau \in {\rm Aut}(S)$ of order $k \geq 2$  such that the quotient orbifold $S/\langle \tau \rangle$ has genus zero;
 we also say that that $\langle \tau \rangle$ is a {\it $k$-gonal group}, and that $S$ is a 
{\it cyclic $k$-gonal curve}.  A $2$-gonal automorphism is the hyperelliptic involution, which is known to be unique.
If, moreover, each fixed point of a non-trivial power of $\tau$ is also a fixed point of $\tau$, then we say that $\tau$ is {\it tame} (for instance, this is the situation if $k$ is a prime integer). If $\tau$ is tame and the rotation angle of $\tau$ at each of its fixed points is the same, then it is called {\it superelliptic}.

Let us fix some $k$-gonal automorphism $\tau$ of $S$. We may identify $S/\langle \tau \rangle$  with the Riemann sphere $\hat{\mathbb C}$.
Let $\pi:S \to \hat{\mathbb C}$ be a Galois branched covering with deck group $\langle \tau \rangle$, and let ${\mathcal B}=\{t_{1},\ldots, t_{n+1}\}$ be its branch locus (this corresponds to the set of cone points of $S/\langle \tau \rangle$). As $S$ has genus at least two, necessarily $n \geq 2$. We note that 
$\tau$ is tame if and only if $\pi$ is fully (or totally) ramified. 
An algebraic (plane) equation for $S$ reflecting the action of $\tau$, is known to be of form \begin{equation}\label{curva0}
{\small S: \; y^{k}=f(x)=(x-t_{1})^{l_{1}} \cdots (x-t_{n+1})^{l_{n+1}},}
\end{equation}
for some non-negative integers $l_{1},\ldots, l_{n+1} \in \{1,\ldots,k-1\}$  (if, for instance, $t_{n+1}=\infty$, then we must delete the factor $(x-t_{n+1})^{l_{n+1}}$). 

Let $N \leq {\rm Aut}(S)$ be the normalizer of $\langle \tau \rangle$, and set $\bar{N}=N/\langle \tau \rangle \leq {\rm PSL}_{2}({\mathbb C})$, which is a subgroup of M\"obius transformations keeping invariant ${\mathcal B}$. Apart from the trivial group, the possibilities for $\bar{N}$ are to be cyclic, dihedral, the alternating groups ${\mathcal A}_{4}$ and ${\mathcal A}_{5}$, or the symmetric group ${\mathfrak S}_{4}$. Moreover, if $n$ is even, then the only possibilities for $\bar{N}$ are the cyclic or dihedral groups. 
The structure of $N$ has been obtained in \cite{BCI,Brandt1,Brandt2} for $k$ prime, in \cite{K,Sanjeewa} for any $k$  and $\tau$ superelliptic, and in \cite{K} for general $k$.

In \cite{HQS}, there were provided a simple necessary and sufficient condition (in terms of $f(x)$ in \eqref{curva0}) for $\tau$ to be central in $N$ (in which case, we say that  $\tau$ is a {\it generalized superelliptic} automorphism). We recall such a condition in Corollary \ref{central}. This condition, in particular, permits us to observe that superelliptic automorphisms are examples of  generalized superelliptic ones.

In many cases, it holds that $N={\rm Aut}(S)$.  This is the situation, for instance, in each of the following cases: (i) $k<(n+1)/2$ is a prime integer \cite{Accola}, or
(ii) $\tau$ is tame, central in $N$ and $k<(n+1)/2$ \cite{K}, or
(iii) $k \geq 5n-2$ is a prime integer \cite{Hidalgo:pgrupo} (in fact, in this case, $\langle \tau \rangle$ is the unique $k$-Sylow subgroup of ${\rm Aut}(S)$).

In \cite{IY}, for  $k<(n+1)/2$ a prime integer, so $N={\rm Aut}(S)$  and $\tau$ is generalized superelliptic, the authors constructed algebraic curves for $S$ reflecting the action of $N$.
In \cite{SS} (see also, \cite{MPRZ}), for $k$ not necessarily prime, but $\tau$ assumed to be superelliptic (so, in particular, tame and central in $N$), the authors constructed algebraic curves for $S$ reflecting the action of $N$. In \cite{HQS2}, it is considered the case when $k$ is arbitrary but $\tau$ is generalized superelliptic (so central in $N$, but not necessarily tame).
In \cite{Wootton1}, for $k$ a prime integer (so $\tau$ is tame, but does not need to be central in $N$), Wootton obtained explicit equations reflecting the action of $N$. The main arguments in that paper are presented within the context of Fuchsian group theory.

In this paper, complementing \cite{Wootton1}, for any integer $k$ and $\tau$ being tame (but not necessarily central in $N$), we provide explicit equations reflecting the action of $N$ (see Section \ref{Sec:ciclico}). In Sections \ref{Sec:234} and \ref{Sec:5}, we explicitly do this for the case when $n=2,3,4,5$. As we will see in Lemma \ref{casoA5}, if $\bar{N} \cong {\mathcal A}_{5}$ and $\tau$ is tame, then $\tau$ is central in $N$, so the equations are already obtained in \cite{IY,SS,Wootton1}. In order to provide the desired description, we proceed as follows. First, note that in the algebraic model \eqref{curva0}, $\tau(x,y)=(x,\omega_{k}y)$, where $\omega_{k}=e^{2\pi i/k}$,  and $\pi(x,y)=x$.  Next, we consider the natural surjective homomorphism
$\theta:N \to \bar{N}$, whose kernel is $\langle \tau \rangle$, satisfying that, for every $\eta \in N$,  $\pi \circ \eta=\theta(\eta) \circ \pi$. 
Let $\eta \in N$, $e=\theta(\eta)$, and $m$ be the order of $e$. Then, 
$\eta(x,y)=(e(x),Q_{\eta}(x) y^{u_{\eta}})$, where $u_{\eta} \in \{1,\ldots,k-1\}$ satisfies that $\eta \tau \eta^{-1} = \tau^{u_{\eta}}$ (so, $u_{\eta}^{m} \equiv 1 \mod(k)$), and $Q_{\eta}(x)^{k}=f(e(x))/f(x)^{u_{\eta}}$ \cite[Lemma 1]{HQS} (Lemma \ref{lemaHQS}). We decompose ${\mathcal B}$ into disjoint subsets ${\mathcal B}_{j}$, where each of these subsets consists of all those branch values of $\pi$ appearing with the same exponent in \eqref{curva0}. The above form of $\eta$ (together with Lemma \ref{observabien}) permits us to observe that $\eta$ permutes these subsets, and allows us to write explicitly the form of $f(x)$ in \eqref{curva0} reflecting such a combinatorial fact, so reflecting the action of $N$.

 We should mention that our arguments are completely algebraic and therefore they can be used for $k$-gonal curves defined over any algebraically closed field of positive characteristic, under the assumption that $k$ is not divisible by the characteristic.

\section{Preliminaries}
In this section, we fix a cyclic $k$-gonal curve $S$, of genus $g \geq 2$, and a $k$-cyclic gonal automorphism $\tau \in {\rm Aut}(S)$.
Let $\pi:S \to \hat{\mathbb C}$ be a Galois branched covering, whose deck group is $\langle \tau \rangle$, and let us 
denote its branch locus by ${\mathcal B}=\{t_{1},\ldots,t_{n+1}\} \subset \hat{\mathbb C}$, where  $n \geq 2$ is an integer. We assume (without loss of generality) that $t_{1},\ldots,t_{n} \in {\mathbb C}$ (so either $t_{n+1} \in {\mathbb C}$ or $t_{n+1}=\infty$). We write $\omega_{k}=e^{2\pi i/k}$ for each integer $k \geq 2.$

\subsection{Algebraic model}\label{Sec:ecuaciondeS}

It is well known that an algebraic model for $S$ is of the form 
\begin{equation}\label{curva}
{\small S: \; y^{k}=f(x)=(x-t_{1})^{l_{1}} \cdots (x-t_{n+1})^{l_{n+1}},}
\end{equation}
where $l_{1},\ldots, l_{n+1} \in \{1,\ldots,k-1\}$, $\gcd(k,l_{1},\ldots,l_{n+1})=1$, and $l_{1}+\cdots+l_{n+1} \equiv 0 \mod(k)$ (if $t_{n+1}=\infty$, then we must delete the factor $(x-t_{n+1})^{l_{n+1}}$ in \eqref{curva}).
In this algebraic model, 
{\small
$\tau(x,y)=(x, \omega_{k}y), \; \mbox{ and } \; \pi(x,y)=x.$
}The branch (cone) order at $t_{j}$ is $k/d_{j}$, where $d_{j}=\gcd(k,l_{j})$, so the signature of $S/\langle \tau \rangle$ is of the form $(0;k/d_{1},\ldots,k/d_{n+1})$ and the genus of $S$ is
{\small $$g=1+\frac{1}{2}\left((n-1)k-(d_1+\cdots+d_{n+1})\right).$$}

\begin{rema}\label{Rem0}\mbox{}
\begin{enumerate}[leftmargin=15pt]
\item If $a \in {\rm PSL}_{2}({\mathbb C})$, then $S$ can also be described by the algebraic curve {\small $y^{k}=(x-a(t_{1}))^{l_{1}} \cdots (x-a(t_{n+1}))^{l_{n+1}}$}. Note that the exponents $l_1, \ldots, l_{n+1}$ are still the same.

\item The $k$-gonal automorphism  $\tau$ is tame if and only if $\gcd(k,l_{j})=1$, for every $j \in \{1,\ldots,n+1\}$, so 
 $g=(k-1)(n-1)/2$.  In particular, $(k-1)(n-1) \geq 4$.
 The set of fixed points of $\tau$ (which is the same for every nontrivial power of $\tau$) is 
${\rm Fix}(\tau)=\{q_{1}=(t_{1},0),\ldots,q_{n+1}=(t_{n+1},0)\}.$
 If $s_{j} \in \{1,\ldots,k-1\}$ is such that $s_{j}l_{j} \equiv 1\mod(k)$, then the rotation angle of $\tau$ about the fixed point  $(t_{j},0)$ if $\tau$ s $\angle(\tau,q_{j})=2\pi s_{j}/k$.
We note that, by replacing $\tau$ by its power $\tau^{s_{1}}$, we may assume that  $l_{1}=1$.

\item $\tau$ is a superelliptic automorphism of order $k$ if and only if $l_{1}=\cdots=l_{n+1}$. In this situation, $\tau$ is central in $N$ (see Corollary \ref{central}).

\item If we replace each $l_{j} \in \{1,\ldots,k-1\}$, in the algebraic model \eqref{curva}, 
by $l_{j}+r_{j}k \in {\mathbb Z}$, where $r_{j} \geq 0$ is any integer, then we obtain a new (singular) algebraic model representing the same curve $S$. We will admit such algebraic models because they allow us to simplify many computations. 

\end{enumerate}
\end{rema}

\subsection{The normalizer subgroup of $\langle \tau \rangle$}
Let $N \leq {\rm Aut}(S)$ be the normalizer of $\langle \tau \rangle$. So, $\bar{N}=N/\langle \tau \rangle \leq {\rm PSL}_{2}({\mathbb C})$ is a finite group that stabilizes the set ${\mathcal B}$.
The finite subgroups of ${\rm PSL}_{2}({\mathbb C})$, apart from the trivial group, are the order $m \geq 2$ cyclic groups ${\mathbb Z}_{m}$, the order $2m$ dihedral groups ${\mathbb D}_{m}$, the alternating groups ${\mathcal A}_{4}$, ${\mathcal A}_{5}$ and the symmetric group ${\mathfrak S}_{4}$.
There is a natural short exact sequence
\begin{equation}\label{short}
1 \to \langle \tau \rangle \to N \stackrel{\theta}{\to} \bar{N} \to 1, \; \mbox{where} \;\pi \circ \eta = \theta(\eta) \circ \pi.
\end{equation}

\begin{rema}
The following facts will be used frequently in our computations.
\begin{enumerate}[leftmargin=15pt]
\item  If $\eta \in N$, then we will denote by $u_{\eta} \in \{1,\ldots, k-1\}$ the value 
such that $\eta \tau \eta^{-1}=\tau^{u_{\eta}}$ (so $\gcd(k,u_{\eta})=1$). If $m$ is the order of $\theta(\eta)$, then 
$u_{\eta}^{m} \equiv 1 \mod(k)$.

\item If $\alpha, \beta \in N$, then $u_{\alpha\beta} \equiv u_{\alpha} u_{\beta} \mod(k)$.
\end{enumerate}
\end{rema}

\begin{lemm}\label{fijado} Let $\eta \in N$ and assume $\tau$ is tame. 
If $\theta(\eta)$ fixes a point in ${\mathcal B}$, then $\eta$ commutes with $\tau$, that is, $u_{\eta}=1$.
\end{lemm}
\begin{proof}
If $e=\theta(\eta)$ fixes a point in ${\mathcal B}$, then there is a $\delta \in N$ such that $\theta(\delta)=e$ and $\delta$ has a common fixed point with $\tau$. As the ${\rm Aut}(S)$-stabilizer of a point is cyclic, it follows that $\delta$ and $\tau$ commute. Since $\eta \delta^{-1} \in \langle \tau \rangle$ (as both $\delta$ and $\eta$ are sent to $e$ under $\theta$), the result follows.
\end{proof}

\begin{rema}
If, in the hypothesis of the above lemma, we delete the condition for $\tau$ to be tame, then we can only ensure that some non-trivial power of $\tau$ commutes with $\eta$.
\end{rema}

The following result describes the algebraic form of those elements in $N$.

\begin{lemm}[Lemma 1 in \cite{HQS}]\label{lemaHQS}
Let $\eta \in N$ and $u_{\eta} \in \{1,\ldots,k-1\}$ be such that  $\eta \tau \eta^{-1}=\tau^{u_{\eta}}$. 
If $e=\theta(\eta)$, and $m$ is the order of $e$, then  
\begin{equation}\label{formaeta}
{\small \eta(x,y)=\left(e(x), Q_{e}(x) y^{u_{\eta}}\right),}
\end{equation}
where
\begin{equation}\label{formaQ}
{\small
Q_{e}(x)^{k}=\frac{f(e(x))}{f(x)^{u_{\eta}}}=\frac{(e(x)-t_{1})^{l_{1}} \cdots (e(x)-t_{n+1})^{l_{n+1}}}{(x-t_{1})^{l_{1}u_{\eta}} \cdots (x-t_{n+1})^{l_{n+1}u_{\eta}}}.
}
\end{equation}
\end{lemm}

\begin{rema}
 In the proof of Lemma \ref{observabien}, the expression of $Q_{e}(x)$ will be given explicitly.
\end{rema}

Let ${\mathcal B}_{1},\ldots,{\mathcal B}_{r}$ be the partition of ${\mathcal B}$ such that each ${\mathcal B}_{j}$ consists of all those cone points having the same exponent in \eqref{curva}. If  $\eta \in N$, and $e=\theta(\eta) \in \bar{N}$, then $e$ permutes the above collection ${\mathcal B}_{1},\ldots,{\mathcal B}_{r}.$ The condition for $\eta$ to commute with $\tau$ (i.e., $u_{\eta}=1$) is equivalent to $e$ keeping fixed each of these sets.

\begin{coro}[\cite{HQS}]\label{central}
The cyclic $n$-gonal automorphism $\tau$ is generalized superelliptic (i.e., central in $N$) if and only if, for every $e \in \bar{N}$, it holds that $e({\mathcal B}_{j})={\mathcal B}_{j}$, for every $j=1,\ldots,r$. 
\end{coro}

\begin{rema}
In particular, if $r=1$ (the case of superelliptic automorphisms), then $\tau$ is central in $N$.
\end{rema}

\begin{rema}[Tame case]\label{deuso}
Let us assume $\tau$ to be tame.
Let $\eta \in N$ and $e=\theta(\eta) \in {\rm PSL}_{2}({\mathbb C})$. 
If $e$ fixes one of the ${\mathcal B}_{j}$, then it fixes each one of the others; so, $\eta$ commutes with $\tau$. This, in particular, holds if one of the sets ${\mathcal B}_{j}$ has cardinality different from the cardinalities of the other sets ${\mathcal B}_{i}$. 
\end{rema}

\begin{rema}
In \cite{AK}, for $\tau$ tame, the authors consider those cases for which there exists an ``extra automorphism", that is, some $\eta \in N \setminus \langle \tau \rangle$, whose order $\delta$ divides $k$. In \cite[Lemma 2.3]{AK}, the authors provide an algebraic form of the curve. In that algebraic model, $\eta(x,y)=(\omega_{\delta}x,y)$ and $\tau=(x,\omega_{k}y)$, so they commute. Note that the cyclic $4$-gonal curve {\small $y^{4}=(x-1)(x+1)^{3}(x-i)(x+i)^{3}$} admits the extra (of order $\delta=2$) automorphism $$\eta(x,y)=(-x,\tfrac{y^{3}}{(x+1)^{2}(x+i)^{2}}) \mbox{ which satisfies } \eta \tau \eta^{-1}=\tau^{-1}$$(this is in contradiction to that lemma). So, probably the condition that $\tau$ is superelliptic is needed (i.e., in the algebraic model \eqref{curva}, the exponents $l_{1}, \ldots, l_{n+1}$ are all equal to one).
\end{rema}

\subsection{The possibilities for the signatures of $S/N$}
As already noticed, the signature of $S/\langle \tau \rangle$ is of the form
$(0;k/d_{1},\ldots,k/d_{n+1})$, where $d_{j}=\gcd(k,l_{j})$. The form of the signatures of $S/N$, in terms of $\bar{N}$, is well known. Below, for $\bar{N}$ different from the trivial group, we recall them (this is obtained by the permutational action of $\bar{N}$ on the partition of ${\mathcal B}$ into the subsets of points with the same exponent $l_{j}$). 

\subsubsection{The cyclic case}
If $\bar{N}=\langle a \rangle \cong {\mathbb Z}_{m}$ ($m \geq 2$), then $|N|=mk$, 
$n+1=s_{0}+ml,$
where $s_{0}\in \{0,1,2\}$ stands for the number of fixed points of $a$ in ${\mathcal B}$ and $l \geq 0$ for the number of orbits in ${\mathcal B}$ of length $m$. In this case, 
$S/N$ has a signature of the form (where $2 \leq k_{j} \leq k$ are divisors of $k$). 
{\small
$$
\begin{array}{ll}
(1.1) \; (0;k_{1},\stackrel{l}{\ldots},k_{l},m,m), & s_{0}=0.\\
(1.2) \; (0;k_{1},\stackrel{l}{\ldots},k_{l},m,mk_{l+1}), & s_{0}=1.\\
(1.3) \; (0;k_{1},\stackrel{l}{\ldots},k_{l},mk_{l+1},mk_{l+2}), & s_{0}=2.
\end{array}
$$
}

Each of the above provides complex $(l-1)$-dimensional families. 

\begin{rema}
If $\tau$ is tame (so, each $k_{j}=k$), then, in cases (1.2) and (1.3), $\tau$ is central in $N$. This follows from the fact that, in these two cases, 
the generator of $\bar{N}$ fixes elements of ${\mathcal B}$ (Lemma \ref{fijado}). In these cases, equations for $S$, reflecting the action of $N$, were obtained in \cite{IY,SS,Wootton1}.
\end{rema}

\subsubsection{The dihedral case}
 If $\bar{N}=\langle a, b: a^{m}=b^{2}=(ab)^{2}=1\rangle \cong {\mathbb D}_{m}$ ($m \geq 2$), then $|N|=2mk$,
$n+1=2s_{0}+m(s_{1}+2l), s_{0} \in \{0,1\}, s_{1} \in \{0,1,2\},  l \geq 0,$
where  $2s_{0}$ (respectively, $ms_{1}$) stands for the number of fixed points of $a$ (respectively, the sum of the number of fixed points of the conjugates of $b$ and $ab$) in ${\mathcal B}$ and $l$ for the number of orbits in ${\mathcal B}$ of length  $2m$.
In this case, $S/N$ has a signature of the form  (where $2 \leq k_{j} \leq k$ are divisors of $k$)
{\small
$$
\begin{array}{ll}
(2.1) \; (0;k_{1},\stackrel{l}{\ldots},k_{l},2,2,m), & s_{0}=s_{1}=0.\\
(2.2) \; (0;k_{1},\stackrel{l}{\ldots},k_{l},2,2k_{l+1},m), & s_{0}=0, s_{1}=1.\\
(2.3) \; (0;k_{1},\stackrel{l}{\ldots},k_{l},2k_{l+1},2k_{l+2},m), & s_{0}=0, s_{1}=2.\\
(2.4) \; (0;k_{1},\stackrel{l}{\ldots},k_{l},2,2,mk_{l+1}), & s_{0}=1, s_{1}=0.\\
(2.5) \; (0;k_{1},\stackrel{l}{\ldots},k_{l},2,2k_{l+1},mk_{l+2}), & s_{0}=1, s_{1}=1.\\
(2.6) \; (0;k_{1},\stackrel{l}{\ldots},k_{l},2k_{l+1},2k_{l+2},mk_{l+3}), & s_{0}=1, s_{1}=2.
\end{array}
$$
}

Each of the above provides complex $l$-dimensional families. 

\begin{rema}
If $\tau$ is tame (so each $k_{j}=k$), then in cases (2.3), (2.5), and (2.6), there are generators of $\bar{N}$ having fixed points in ${\mathcal B}$, so $\tau$ is central in $N$. In these cases, equations for $S$, reflecting the action of $N$, are similar to the ones obtained in \cite{IY,SS,Wootton1}. 
\end{rema}

\subsubsection{The alternating case ${\mathcal A}_{4}$} 
If $\bar{N}=\langle a,b: a^{3}=b^{2}=(ab)^{3}=1\rangle \cong {\mathcal A}_{4}$, then $|N|=12k$, 
$n+1=4s_{0}+6s_{1}+12l,$ $s_{0} \in \{0,1,2\},$ $s_{1} \in \{0,1\},$  $l \geq 0,$
where $4s_{0}$ (respectively, $6s_{1}$) stands for the sum of fixed points of the conjugates of $a$ (respectively, the sum of fixed points of the conjugates of $b$) in ${\mathcal B}$, and $l$ for the number of orbits in ${\mathcal B}$ of length $12$.
 In this case, $S/N$ has a signature of the form (where $2 \leq k_{j} \leq k$ are divisors of $k$)
{\small
$$
\begin{array}{ll}
(3.1) \; (0;k_{1},\stackrel{l}{\ldots},k_{l},2,3,3), & s_{0}=s_{1}=0.\\
(3.2) \; (0;k_{1},\stackrel{l}{\ldots},k_{l},2k_{l+1},3,3), & s_{0}=0, s_{1}=1.\\
(3.3) \; (0;k_{1},\stackrel{l}{\ldots},k_{l},2,3,3k_{l+1}), & s_{0}=1, s_{1}=0.\\
(3.4) \; (0;k_{1},\stackrel{l}{\ldots},k_{l},2k_{l+1},3,3k_{l+2}), & s_{0}=1, s_{1}=1.\\
(3.5) \; (0;k_{1},\stackrel{l}{\ldots},k_{l},2,3k_{l+1},3k_{l+2}), & s_{0}=2, s_{1}=0.\\
(3.6) \; (0;k_{1},\stackrel{l}{\ldots},k_{l},2k_{l+1},3k_{l+2},3k_{l+3}), & s_{0}=2, s_{1}=1.
\end{array}
$$
}

Each of the above provides complex $l$-dimensional families.

\begin{rema}
If $\tau$ is tame, then in cases (3.4), (3.5), and (3.6), there are generators of $\bar{N}$ having fixed points in ${\mathcal B}$, so $\tau$ is central in $N$. In these cases, equations for $S$, reflecting the action of $N$, are similar to the ones obtained in \cite{IY,SS,Wootton1}. 
\end{rema}

\subsubsection{The symmetric case ${\mathfrak S}_{4}$}
If $\bar{N}=\langle a,b: a^{4}=b^{2}=(ab)^{3}=1\rangle \cong {\mathfrak S}_{4}$, then $|N|=24k$, 
$n+1=6s_{0}+12s_{1}+8s_{2}+24l,$ $s_{0},s_{1},s_{2} \in \{0,1\},$ $l \geq 0,$
where $6s_{0}$ (respectively, $12s_{1}$ and $8s_{2}$) stands for the sum of fixed points of the conjugates of $a$ (respectively, the conjugates of $b$ and  $ab$) in ${\mathcal B}$, and $l$ for the number of orbits in ${\mathcal B}$ of length $24$.
 In this case,  $S/N$ has a signature of the form (where $2 \leq k_{j} \leq k$ are divisors of $k$)
{\small
$$
\begin{array}{ll}
(4.1) \; (0;k_{1},\stackrel{l}{\ldots},k_{l},2,3,4), & s_{0}=s_{1}=s_{2}=0.\\
(4.2) \; (0;k_{1},\stackrel{l}{\ldots},k_{l},2,3k_{l+1},4), & s_{0}=0, s_{1}=0, s_{2}=1.\\
(4.3) \; (0;k_{1},\stackrel{l}{\ldots},k_{l},2k_{l+1},3,4), & s_{0}=0, s_{1}=1, s_{2}=0.\\
(4.4) \; (0;k_{1},\stackrel{l}{\ldots},k_{l},2k_{l+1},3k_{l+2},4), & s_{0}=0, s_{1}=s_{2}=1.\\
(4.5) \; (0;k_{1},\stackrel{l}{\ldots},k_{l},2,3,4k_{l+1}), & s_{0}=1, s_{1}=s_{2}=0.\\
(4.6) \; (0;k_{1},\stackrel{l}{\ldots},k_{l},2,3k_{l+1},4k_{l+2}), & s_{0}=1, s_{1}=0, s_{2}=1.\\
(4.7) \; (0;k_{1},\stackrel{l}{\ldots},k_{l},2k_{l+1},3,4k_{l+2}), & s_{0}=1, s_{1}=1, s_{2}=0.\\
(4.8) \; (0;k_{1},\stackrel{l}{\ldots},k_{l},2k_{l+1},3k_{l+2},4k_{l+3}), & s_{0}=1, s_{1}=s_{2}=1.
\end{array}
$$
}

Each of the above provides complex $l$-dimensional families. 

\begin{rema}
If $\tau$ is tame (so, $k_{j}=k$), then in cases (4.4), (4.6), (4.7), and (4.8), there are generators of $\bar{N}$ having fixed points in ${\mathcal B}$, so $\tau$ is central in $N$. In these cases, equations for $S$, reflecting the action of $N$, are similar to the ones obtained in \cite{IY,SS,Wootton1}.
\end{rema}

\subsubsection{The alternating case ${\mathcal A}_{5}$} 
If $\bar{N}=\langle a,b: a^{5}=b^{2}=(ab)^{3}=1\rangle \cong{\mathcal A}_{5}$, then $|N|=60k$,
$n+1=12s_{0}+30s_{1}+20s_{2}+60l,$ $s_{0},s_{1},s_{2} \in \{0,1\},$ $l \geq 0,$
where $12s_{0}$ (respectively, $30s_{1}$ and $20s_{2}$) stands for the sum of fixed points of the conjugates of $a$ (respectively, the conjugates of $b$ and  $ab$) in ${\mathcal B}$, and $l$ for the number of orbits in ${\mathcal B}$ of length $60$.
 In this case,  $S/N$ has a signature of the form (where $2 \leq k_{j} \leq k$ are divisors of $k$)
{\small
$$
\begin{array}{ll}
(5.1) \; (0;k_{1},\stackrel{l}{\ldots},k_{l},2,3,5), & s_{0}=s_{1}=s_{2}=0.\\
(5.2) \; (0;k_{1},\stackrel{l}{\ldots},k_{l},2,3k_{l+1},5), & s_{0}=0, s_{1}=0, s_{2}=1.\\
(5.3) \; (0;k_{1},\stackrel{l}{\ldots},k_{l},2k_{l+1},3,5), & s_{0}=0, s_{1}=1, s_{2}=0.\\
(5.4) \; (0;k_{1},\stackrel{l}{\ldots},k_{l},2k_{l+1},3k_{l+2},5), & s_{0}=0, s_{1}=s_{2}=1.\\
(5.5) \; (0;k_{1},\stackrel{l}{\ldots},k_{l},2,3,5k_{l+1}), & s_{0}=1, s_{1}=s_{2}=0.\\
(5.6) \; (0;k_{1},\stackrel{l}{\ldots},k_{l},2,3k_{l+1},5k_{l+2}), & s_{0}=1, s_{1}=0, s_{2}=1.\\
(5.7) \; (0;k_{1},\stackrel{l}{\ldots},k_{l},2k_{l+1},3,5k_{l+2}), & s_{0}=1, s_{1}=1, s_{2}=0.\\
(5.8) \; (0;k_{1},\stackrel{l}{\ldots},k_{l},2k_{l+1},3k_{l+2},5k_{l+3}), & s_{0}=1, s_{1}=s_{2}=1.
\end{array}
$$
}

Each of the above provides complex $l$-dimensional families.
In each case, $\tau$ is central in $N$ (see Lemma \ref{casoA5} below). So, equations for $S$, reflecting the action of $N$, are similar to the ones obtained in \cite{IY,SS,Wootton1}. 

\begin{lemm}\label{casoA5}
If $\bar{N} \cong {\mathcal A}_{5}$, then $\tau$ is central in $N$. 
\end{lemm}
\begin{proof}
Let $\bar{N}=\langle a,b: a^{5}=b^{2}=(ab)^{3}=1\rangle$. Let $\alpha,\beta \in N$ be such that $\theta(\alpha)=a$ and $\theta(\beta)=b$.
Then, $u_{\alpha}^{5} \equiv 1 \mod(k)$, $u_{\beta}^{2} \equiv 1 \mod(k)$, and $(u_{\alpha}u_{\beta})^{3} \equiv 1 \mod(k)$.
The last asserts that $u_{\alpha}^{3} \equiv u_{\beta} \mod(k)$, so $u_{\alpha} \equiv u_{\alpha}^{6} \equiv u_{\beta}^{2} \equiv 1 \mod(k)$. It follows that $u_{\alpha}=u_{\beta}=1$, i.e., $\alpha$ and $\beta$ conmute with $\tau$.
\end{proof}

\section{The algebraic descriptions of $S$ in terms of $\bar{N}$: case $\tau$ tame}\label{Sec:ciclico}
In this section, we assume $\tau$ to be a tame $k$-gonal automorphism and $S$ to be described by an algebraic curve as in \eqref{curva}. As previously noticed, when $n \geq 2$ is even, the only possibilities for $\bar{N}$ are the cyclic and the dihedral groups. 
Below, we provide such an algebraic description for $S$, in terms of $\bar{N}$, reflecting the action of $N$. As already mentioned, 
this has been done in \cite{Wootton1} for $k=p$ a prime integer (see \cite[Theorem 7.5.]{Wootton1}). As, for $\bar{N} \cong {\mathcal A}_{5}$, $\tau$ is necessarily central in $N$ (see Lemma \ref{casoA5}), these equations were obtained in \cite{IY,SS,Wootton1}, so we will not take care of this one.

Let us recall the natural short exact sequence
$1 \to \langle \tau \rangle \to N \stackrel{\theta}{\to} \bar{N} \to 1,$
and the fact that, for $\eta \in N$, and $e=\theta(\eta)$, one has that $\eta(x,y)=(e(x),Q_{e}(x)y^{u_{\eta}})$, where $u_{\eta} \in \{1,\ldots,k-1\}$ satisfies that $\eta \circ \tau \circ \eta^{-1}=\tau^{u_{\eta}}$ (so
$u_{\eta}^{m} \equiv 1 \mod(k)$, where $m$ is the order of $e$), and the rational map $Q_{e}(x)$ can be explicitly computed from the form of $f(x)$ (see \eqref{formaQ}); this is explicitly given in \eqref{formaQ1} below.

\subsection{A computational remark}
The following result, which will be frequently used for computations of curves and automorphisms, allows us to simplify the determination of the integers $l_{j}$ in \eqref{curva}. In the case that $k=p$ is a prime integer, this was obtained in \cite[Lemma 7.4.]{Wootton1}, whose proof used the theory of Fuchsian groups. As our arguments are algebraic, they can be used for curves defined over any algebraically closed field of any characteristic as long $k$ is relatively prime to the characteristic (in the positive characteristic situation). 

\begin{lemm}\label{observabien}
Let $\tau$ be tame $k$-gonal automorphism, $\eta \in N$, different from the identity, and $m \geq 2$ be the order of $a=\theta(\eta) \in \bar{N}$. 
Let $s_{1}=t_{i_{1}}, \ldots, s_{r}=t_{i_{r}}$ (where $r \in \{1,m\}$) be a full $a$-orbit of points of ${\mathcal B}$, say $a(s_{j})=s_{j+1}$ (for $j=1,\ldots,r-1$) and $a(s_{r})=s_{1}$. Then $l_{i_{j+1}} \equiv u_{\eta}^{j} l_{i_{1}} \mod(k)$ (for $j=1,\ldots,r-1$). 

\end{lemm}

\begin{rema}
If $a$ fixes a point in ${\mathcal B}$, then the above lemma ensures that $u_{\eta}=1$ (see Lemma \ref{fijado}).
\end{rema}

\begin{proof}
A full $a$-orbit in ${\mathcal B}$ is either: (i) of length one (if one of both fixed points of $a$ belongs to ${\mathcal B}$), or one of length $m$ (whose elements are cyclically permuted by $a$).

(1) Let us first consider a full $a$-orbit $s_{1}=t_{i_{1}}, \ldots, s_{r}=t_{i_{r}}$ (where $r \in \{1,m\}$) of points of ${\mathcal B} \cap {\mathbb C}$. We may assume that $a(s_{j})=s_{j+1}$ (for $j=1,\ldots,r-1$) and $a(s_{r})=s_{1}$. Let us denote by $l_{i_{j}}$ the exponent of $t_{i_{j}}$ in \eqref{curva}.
There is some $v \in \{1,\ldots,k-1\}$, such that $l_{i_{j+1}} \equiv v^{j} l_{i_{1}} \mod(k)$ (for $j=1,\ldots,r-1$), and $l_{1} \equiv v^{r} l_{1} \mod(k)$. In particular, $(v^{r}-1)l_{1} \equiv 0 \mod(k)$.
So, if $\tau$ is tame, then (as $\gcd(k,l_{1})=1$) $v$ satisfies that $v^{r} \equiv 1 \mod(k)$.
If we set $l_{i_{1}}=\hat{l}$, then one of the factors of $f(x)$ is of the form 
{\small
$$g(x)=(x-s_{1})^{\hat{l}}(x-s_{2})^{v\hat{l}}\cdots(x-s_{r})^{v^{r-1}\hat{l}}.$$ 
}
Since $a(x)-a(s)=(a'(x)a'(s))^{1/2}(x-s)$, if we set 
$\hat{\lambda}=a'(s_{r})^{\hat{l}/2} a'(s_{1})^{\hat{l}v/2} \cdots a'(s_{r-1})^{\hat{l}v^{r-1}/2}$, then 
{\small
$$g(a(x))=
(a(x)-s_{1})^{\hat{l}}(a(x)-s_{2})^{v\hat{l}}\cdots(a(x)x-s_{r})^{v^{r-1}\hat{l}}=$$
$$(a(x)-a(s_{r}))^{\hat{l}}(a(x)-a(s_{1}))^{v\hat{l}}\cdots(a(x)x-a(s_{r-1}))^{v^{r-1}\hat{l}}=
$$
$$
=\hat{\lambda} a'(x)^{\frac{\hat{l}(1+v+\cdots+v^{r-1})}{2}}(x-s_{r})^{\hat{l}}(x-s_{1})^{v\hat{l}}\cdots(x-s_{r-1})^{v^{r-1}\hat{l}} =$$
$$
=\hat{\lambda} a'(x)^{\frac{\hat{l}(1+v+\cdots+v^{r-1})}{2}} \frac{(x-s_{r})^{v^{r}\hat{l}}(x-s_{1})^{v \hat{l}}\cdots(x-s_{r-1})^{v^{r-1}\hat{l}}}{(x-s_{r})^{(v^{r}-1)\hat{l}}}=
\hat{\lambda} a'(x)^{\frac{\hat{l}(1+v+\cdots+v^{r-1})}{2}} \frac{g(x)^{v}}{(x-s_{r})^{(v^{r}-1)\hat{l}}}.
$$
}

(2) Let us consider a maximal subset $\{\hat{s}_{r_{1}},\ldots, \hat{s}_{r_{q}}\} \subset {\mathcal B}$, where they are in different $a$-orbits. So, by the above, there exist 
\begin{enumerate}[leftmargin=15pt]
\item $r_{1}, \ldots,r_{q} \in \{1,m\}$, and at most two of them can be equal to $1$, 
\item $\hat{l}_{1},\ldots,\hat{l}_{q} \in \{1,\ldots,k-1\}$ (these are some of the exponents in \eqref{curva}), 
\item $v_{1},\ldots,v_{q} \in \{1,\ldots, k-1\}$ such that $(v_{j}^{r_{j}}-1)\hat{l}_{j} \equiv 0 \mod(k)$. So, if $\tau$ is tame, then $v_{j}^{r_{j}} \equiv 1 \mod(k)$,
 and
 \item $g_{1}(x),\ldots,g_{q}(x) \in {\mathbb C}[x]$, where each $g_{j}(x)$ is a factor coming from the $a$-orbit of $\hat{s}_{j}$,
\end{enumerate}
such that 
$f(x)=g_{1}(x) \cdots g_{q}(x)$, and 
{\small
$$
f(a(x))=\hat{\lambda} 
a'(x)^{\frac{\hat{l}_{1}(1+v_{1}+\cdots+v_{1}^{r_{1}-1})+\cdots+\hat{l}_{q}(1+v_{q}+\cdots+v_{q}^{r_{q}-1})}{2}} 
\frac{g_{1}(x)^{v_{1}} \cdots g_{q}(x)^{v_{m}}}{(x-s_{r_{1}})^{(v_{1}^{r_{1}}-1){\hat{l}_{1}} } \cdots(x-s_{r_{q}})^{(v_{q}^{r_{q}}-1){\hat{l}_{q}}}},
$$}where $\hat{\lambda}=\hat{\lambda}_{1}+\cdots+\hat{\lambda}_{q}$ is a non-zero constant. 
Note that the exponents of $a'(x)$ and of $(x-s_{j})$ are integer multiples of $k$. We know that $\eta(x,y)=(a(x), Q_{a}(x) y^{u_{\eta}})$, where
{\small
$$Q_{a}(x)^{k}=\frac{f(a(x))}{f(x)^{u_{\eta}}}=\left(\frac{\hat{\lambda} 
a'(x)^{\frac{\hat{l}_{1}(1+v_{1}+\cdots+v_{1}^{r_{1}-1})+\cdots+\hat{l}_{q}(1+v_{q}+\cdots+v_{q}^{r_{q}-1})}{2}} }{(x-s_{r_{1}})^{(v_{1}^{r_{1}}-1){\hat{l}_{1}} } \cdots(x-s_{r_{q}})^{(v_{q}^{r_{q}}-1){\hat{l}_{q}}}}\right)
\frac{1}{ g_{1}(x)^{u_{\eta}-v_{1}} \cdots g_{q}(x)^{u_{\eta}-v_{m}}}.$$
}

The above asserts that 
$g_{1}(x)^{u_{\eta}-v_{1}} \cdots g_{q}(x)^{u_{\eta}-v_{m}}$ is the $k$-power of a polynomial. This implies that
$(u_{\eta}-v_{j})\hat{l}_{j} \equiv  0 \mod(k)$. So, if we assume $\tau$ to be tame, then $u_{\eta}-v_{j} \equiv  0 \mod(k)$, and therefore we must have that $v_{1}=\cdots=v_{q}=u_{\eta}$, and 
\begin{equation}\label{formaQ1}
{\small Q_{a}(x)=\frac{\hat{\lambda}^{1/k} a'(x)^{\frac{\hat{l}_{1}(1+u_{\eta}+\cdots+u_{\eta}^{r_{1}-1})+\cdots+\hat{l}_{q}(1+u_{\eta}+\cdots+u_{\eta}^{r_{q}-1})}{2k}}}
{(x-s_{r_{1}})^{\frac{(u_{\eta}^{r_{1}}-1){\hat{l}_{1}}}{k} } \cdots(x-s_{r_{q}})^{\frac{(u_{\eta}^{r_{q}}-1){\hat{l}_{q}}}{k}}}.}
\end{equation}

In particular, if $a$ fixes one of the elements of ${\mathcal B}$ (so $r=1$ in this case) and $\tau$ is tame, then $u_{\eta}=v_{1}=1$ (see part (1) in Remark \ref{deuso}).
\end{proof}

\begin{rema}
In the above lemma, the hypothesis for $\tau$ to be tame is necessary. For example, let us consider the curve
{\small $y^{4}=x^{2}(x^{2}-1)(x^{2}+1)^{3},$}
for which $\tau$ is not tame. This curve admits the automorphism
$\eta \in N$ defined by 
$\eta(x,y)=\left(ix, \frac{\omega_{8} y^{3}}{x(x^{2}+1)^{2} } \right).$ In this case, $u_{\eta}=3$ (since $\eta \tau \eta^{-1}=\tau^{3}$), and $a(x)=\theta(\eta)(x)=ix$.
There are three full $a$-orbits in ${\mathcal B}=\{\infty,0,\pm1, \pm i\}$, given by $\{\infty\}$, $\{0\}$ and $\{1,i,-1,-i\}$. We note that $a(0)=0$, but $u_{\eta} \neq 1$.

\end{rema}

\subsection{The case when $\bar{N}$ contains ${\mathbb Z}_{m}$}
Let us consider the case when $\bar{N}$ contains ${\mathbb Z}_{m}$, where $m \geq 2$. Note that $n+1=s_{0}+ml$, where $s_{0} \in \{0,1,2\}$ and $l \geq 1$, and therefore there are three possibilities in this case. We may assume ${\mathbb Z}_{m}$ is generated by $a(x)=\omega_{m}x$ and that the cone points are given by $\omega_{m}^{j}, \omega_{m}^{j}\lambda_{1}, \ldots, \omega_{m}^{j}\lambda_{l-1}$, where $j \in \{0,\ldots,m-1\}$, $\lambda_{i}^{m} \neq 0,1$ and $\lambda_{i_{1}}^{m} \neq \lambda_{i_{2}}^{m}$ for $i_{1} \neq i_{2}$, and perhaps $0$ and/or $\infty$.  
In this case, if $\alpha \in N$ is such that $\theta(\alpha)=a$, then $\alpha(x,y)=(\omega_{m}x, Q(x)y^{u})$, $u=u_{\alpha} \in \{1,\ldots,k-1\}$, $\gcd(k,u)=1$ and $u^{m} \equiv 1 \mod(k)$. In particular, $\alpha \tau \alpha^{-1}=\tau^{u}$.

\subsubsection{\bf The case $s_{0}=0$}
If $ml=n+1$ and $S/\theta^{-1}({\mathbb Z}_{m})$ of signature $(0;k,\stackrel{l}{\ldots},k ,m,m)$. In this case, 
{\small 
$$S: y^{k}=\prod_{j=0}^{m-1}(x-\omega_{m}^{j})^{u^{j}} \prod_{i=1}^{l-1} \left( \prod_{j=0}^{m-1}(x-\omega_{m}^{j}\lambda_{i})^{u^{j}}  \right)^{t_{i}},$$
}
where $t_{i} \in \{1,\ldots,k-1\}$, $\gcd(k,t_{i})=1$, and $(1+u+u^{2}+\cdots+u^{m-1})(1+t_{1}+\cdots+t_{l-1}) \equiv 0 \mod(k)$, and we may assume 
{\small
$$\alpha(x,y)=\left( \omega_{m}x,\frac{\omega_{m}^{\frac{(1+u+\cdots+u^{m-1})(1+t_{1}+\cdots+t_{l-1})}{k}} y^{u} }
{\left( (x-\omega_{m}^{m-1})  \prod_{i=1}^{l-1}(x-\omega_{m}^{m-1}\lambda_{i})\right)^{\frac{u^{m}-1}{k}} }\right).$$
}

Note that if $\alpha^{m} \neq 1$, then $\hat{\alpha}=\alpha \tau^{1+t_{1}+\cdots+t_{l-1}}$ satisfies that $\hat{\alpha}^{m}=1$. In this situation, $N$ contains the subgroup
$$\langle \tau, \alpha: \tau^{k}=\alpha^{m}=1, \alpha \tau \alpha^{-1}=\tau^{u}\rangle=\langle \tau \rangle \rtimes \langle \alpha \rangle \cong {\mathbb Z}_{k} \rtimes {\mathbb Z}_{m}.$$

\begin{rema}
For $u=1$, the above equations were provided in \cite[Table 4]{SS}. In the generic situation, $N=\langle \tau, \alpha \rangle$.
\end{rema}

\begin{rema}
In the above, $\alpha$ commutes with $\tau$ if and only if $m(1+t_{1}+\cdots+t_{l-1}) \equiv 0 \mod(k)$, in which case, 
{\small $S: y^{k}=(x^{m}-1)\prod_{i=1}^{l-1}(x^{m}-\lambda_{i}^{m})^{t_{i}}$} and $\alpha(x,y)=(\omega_{m}x, y).$
In the particular case when $k=p$ is a prime integer, this corresponds to having either (i) $m=p$  or (ii) $m \neq p$ and $1+t_{1}+\cdots+t_{l-1} \equiv 0 \mod(p)$.
\end{rema}

\begin{rema}
Note that for some particular choices of $n$ and $k$, there will be situations for which $\bar{N}$ equals ${\mathbb Z}_{m}$ and others for which $\bar{N}$ contains strictly ${\mathbb Z}_{m}$ 
(usually it is a dihedral group). Let us consider a couple of cases.
\begin{enumerate}[leftmargin=15pt]
\item If $k=7$, $l=1$, $u=2$, $n=2$ and $m=3$, then 
{\small $S: y^{7}=(x-1)(x-\omega_{3})^{2}(x-\omega_{3}^{2})^{4},$} $\bar{N}=\langle a(x)=\omega_{3}x \rangle \cong {\mathbb Z}_{3}$, and 
$N=\left< \tau(x,y)=(x,\omega_{7}y),\;\alpha(x,y)=\left(\omega_{3}x, \frac{\omega_{3}y^{2}}{x-\omega_{3}^{2}}\right)\right>.$ In this case, $\alpha \tau \alpha^{-1}=\tau^{2}$.

\item If $m=n+1 \equiv 0 \mod(k)$, $l=1$ and $u=1$, then   
{\small $S: y^{k}=x^{n+1}-1,$ }
$\bar{N}=\left< a(x)=\omega_{n+1}x,b(x)=\frac{1}{x}\right> \cong {\mathbb D}_{n+1},$ and 
$N=\left< \tau(x,y)=(x,\omega_{k}y),\; \alpha(x,y)=(\omega_{n+1}x,y),\; \beta(x,y)=\left(\frac{1}{x}, \frac{(-1)^{1/k}y}{x^{l}}\right)\right>.$ In this case, $\alpha$ and $\beta$ commute with $\tau$.
\end{enumerate}

\end{rema}

\subsubsection{\bf The case $s_{0}=1$} If $ml=n$ and $S/\theta^{-1}({\mathbb Z}_{m})$ of signature $(0,k,\stackrel{l}{\dots},k,m,mk)$.
As $a$ fixes an element of ${\mathcal B}$, we have that $u=1$, that is, $\alpha \tau \alpha^{-1}=\tau$.
In this case, 
{\small
$$S: y^{k}=x (x^{m}-1)^{t_{1}} \prod_{i=1}^{l-1} (x^{m}-\lambda_{i}^{m})^{t_{i+1}} ,$$
}
where $t_{i} \in \{1,\ldots,k-1\}$, $\gcd(k,t_{i})=1$ and $1+m(t_{1}+\cdots+t_{l}) \equiv 0 \mod(k),$
and may assume
{\small
$$\alpha(x,y)=\left( \omega_{m}x,\omega_{m}^{\frac{1}{k}} y\right).$$
}

In this case, $\alpha^{m}=\tau$. Note that $\hat{\alpha}=\alpha \tau^{t_{1}+\cdots+t_{l}}$ satisfies that $\hat{\alpha}^{m}=1$. The group $N$ contains the subgroup
$$\langle \tau, \alpha\rangle =\langle \alpha \rangle \cong {\mathbb Z}_{km}.$$

\begin{rema}
As $u=1$, the above equations were provided in \cite[Table 4]{SS}. In the generic situation, $N=\langle \tau, \alpha \rangle$. 
\end{rema}

\begin{rema}
Again, there will be many cases for which $\bar{N}={\mathbb Z}_{m}$, but also some cases for which $\bar{N}$ strictly contains ${\mathbb Z}_{m}$. For instance, 
if we take $k=4$ and $l=t_{1}=1$, then {\small $S: y^{4}=x(x^{3}-1),$} and 
$$\bar{N}=\left< a(x)=\omega_{3}x, b(x)=\frac{\omega_{3}(1-x)}{1+(-1+\omega_{3}-\omega_{3}^{2})x}\right> \cong {\mathcal A}_{4}.$$
\end{rema}

\subsubsection{\bf The case $s_{0}=2$} If $ml=n-1$ and $S/\theta^{-1}({\mathbb Z}_{m})$ of signature $(0;k,\stackrel{l}{\ldots},k,mk,mk)$.
Again as $a$ fixes an element of ${\mathcal B}$, $u=1$ and $\alpha \tau \alpha^{-1}=\tau$.
In this case, 
{\small
$$S: y^{k}=x (x^{m}-1)^{t_{1}} \prod_{i=1}^{l-1} (x^{m}-\lambda_{i}^{m})^{t_{i+1}} ,$$
}
where $t_{i} \in \{1,\ldots,k-1\}$, $\gcd(k,t_{i})=1$ and $\gcd(k,1+m(t_{1}+\cdots+t_{l}))=1$, and we may assume 
{\small
$$\alpha(x,y)=\left( \omega_{m}x,\omega_{m}^{\frac{1}{k}} y\right).$$
}

In this case, $\alpha^{m}=\tau$. For every integer $s \geq 0$ it holds that $(\alpha \tau^{s})^{m}=\tau^{1+sm}$. So, if there is no $s$ such that $1+sm \equiv 0 \mod(k)$, then we cannot find another choice for $\alpha$ of order $m$. In any case, the group $N$ contains the subgroup 
$$\langle \tau, \alpha\rangle = \langle \alpha \rangle \cong {\mathbb Z}_{km}.$$

\begin{rema}
Again, as $u=1$, the above equations were also provided in \cite[Table 4]{SS}. In the generic situation, $N=\langle \tau, \alpha \rangle$. 
\end{rema}

\begin{rema}
Let us observe that if $l=1$, i.e., $m=n-1$, and $2+(n-1)t \equiv 0 \mod(k),$ then the curve also admits the extra automorphism 
$\beta(x,y)=\left(\frac{1}{x}, \frac{(-1)^{1/k}y}{x^{\frac{2+(n-1)t}{k}}}\right) \in N.$
In this case, $$\bar{N}=\langle a(x)=\omega_{n-1}x, b(x)=1/x \rangle= {\mathbb D}_{n-1}.$$
\end{rema}

\begin{rema}
The above computations permit us to observe that, if $\eta \in N$ is such that $\theta(\eta)$ does not fix a point in ${\mathcal B}$, then we can assume $\eta$ and $\theta(\eta)$ to have the same order (see \cite[Theorem 9]{K}).
\end{rema}

\subsection{The case when $\bar{N}$ contains ${\mathbb D}_{m}$}
Let us now assume $\bar{N}$ contains ${\mathbb D}_{m}$, where $m \geq 2$. We may assume 
${\mathbb D}_{m}=\langle a(x)=\omega_{m}x, b(x)=1/x \rangle$.
If $\alpha, \beta \in N$ are such that $\theta(\alpha)=a$ and $\theta(\beta)=b$, then (taking $u=u_{\alpha}$ and $v=u_{\beta}$)
{\small
$$\alpha(x,y)=\left(\omega_{m}x, Q_{1}(x) y^{u}\right), \; \beta(x,y)=\left(\frac{1}{x},Q_{2}(x) y^{v} \right), \mbox{ where }Q_{1}(x)=\left(\frac{f(\omega_{m}x)}{f(x)^{u}}  \right)^{1/k}, \;   Q_{2}(x)=\left(\frac{f(1/x)}{f(x)^{v}}  \right)^{1/k}.$$
}

So, $u^{m} \equiv 1 \mod(k)$ and $v^{2} \equiv 1 \mod(k)$.
Since $(ab)^{2}=1$, we must also have that $(uv)^{2} \equiv 1 \mod(k)$, that is, $u^{2}=1$. In particular, $u^{\gcd(2,m)} \equiv 1 \mod(k)$. In this case, $S$ is represented by a curve of the form
{\small
$$S: y^{k}=f(x)=x^{s_{0}} \left(\prod_{j=0}^{m-1}(x-\omega_{m}^{j})^{u^{j}}\right)^{s_{1}}
\left(\prod_{j=0}^{m-1}(x-\omega_{m}^{j} \omega_{2m})^{u^{j}}\right)^{s_{2}} 
\prod_{i=1}^{l} \left(\prod_{j=0}^{m-1}(x-\omega_{m}^{j} \lambda_{i})^{u^{j}}  (x-\omega_{m}^{j} \lambda_{i}^{-1})^{vu^{j}}\right)^{r_{i}}
,$$
}
where 
\begin{enumerate}[leftmargin=15pt]
\item $l \geq 0$,
\item $u,v \in \{1,\ldots,k-1\}$, $u^{m} \equiv u^{2} \equiv v^{2} \equiv 1 \mod(k)$, 
\item $s_{0} \in \{0,1\}$, and if $s_{0}=1$, then $u=1$,
\item $s_{1},s_{2} \in \{0,1,\ldots,k-1\}$, and if some $s_{j} \neq 0$, then $\gcd(k,s_{j})=1$ and $v=1$,
\item if $l \geq 1$, then $r_{i} \in \{1,\ldots,k-1\}$, and $\gcd(k,r_{i})=1$, $\lambda_{i}^{k} \neq 0,-1,1$, and $\lambda_{i_{1}}^{k} \neq \lambda_{i_{2}}^{\pm k}$,
\item $(1+v)s_{0}+(s_{1}+s_{2}+(1+v)\sum_{i=1}^{l}r_{i})\sum_{j=1}^{m-1}u^{j} \equiv 0 \mod(k).$
\end{enumerate}

\begin{rema}
If $u=v=1$, then the above equations were also provided in \cite[Table 4]{SS}. In the generic situation, $N=\langle \tau, \alpha, \beta \rangle$.  
\end{rema}

\begin{exem}
For instance, if $s_{0}=s_{1}=s_{2}=0$, then
{\small
$$Q_{1}(x)=\frac{\omega_{m}^{\frac{(1+v) \left(\sum_{i=1}^{l}r_{i}\right) \left(\sum_{j=0}^{m-1}u^{j}\right)}{k}}}
{ \left(\prod_{i=1}^{l}(x-\omega_{m}^{m-1} \lambda_{i})^{r_{i}}(x-\omega_{m}^{m-1}\lambda_{i}^{-1})^{r_{i}}\right)^{\frac{u^{m}-1}{k}} },$$
$$Q_{2}(x)=
\frac{(-1)^{\frac{(1+v)\left(\sum_{i=1}^{l}r_{i}\right)\left(\sum_{j=0}^{m-1}u^{j}\right)}{k} } \omega_{m}^{\frac{(1+v) \left(\sum_{i=1}^{l}r_{i}\right) \left(\sum_{j=0}^{m-1}ju^{j}\right)}{k}}
\prod_{i=1}^{l} \lambda_{i}^{\frac{(1-v) \sum_{j=0}^{m-1}u^{j}}{k}}}
{x^{\frac{(1+v)\left(\sum_{i=1}^{l}r_{i}\right) \left(\sum_{j=0}^{m-1}u^{j}  \right)}{k}}
\prod_{i=1}^{l}\left(\prod_{j=0}^{m-1}(x-\omega_{m}^{m-j}\lambda_{i})^{\frac{vu^{j}(u^{m-2j}-1)}{k}} (x-\omega_{m}^{m-j}\lambda_{i}^{-1})^{\frac{u^{m-j}(v^{2}-1)}{k}}  \right)
}
$$
}
\end{exem}

\begin{exem}[Maximal dihedral]\label{maxdihe}
Assume $m=n+1$, where $n \geq 2$ is even. In this case, $g=(k-1)(n-1)/2$, so $k \geq 3$ must be odd, and $u=1$. In this case, $l=0$, $s_{0}=s_{2}=0$, and may assume $s_{1}=1$. As $b$ fixes $1 \in {\mathcal B}$, it also follows that $v=1$.
So, $S: y^{k}=x^{n+1}-1,$
where $k/(n+1)$. For example, $n=2$ is not possible; as in such a case, $k=3$ and $g=1$, a contradiction.
\end{exem}

\begin{rema}
If $m \geq 3$ is odd, then $u=1$, that is, $\alpha$ commutes with $\tau$, and 
{\small $S: y^{k}=x^{s_{0}}(x^{m}-1)^{s_{1}}(x^{m}+1)^{s_{2}} \prod_{i=1}^{l}\left((x^{m}-\lambda_{i}^{m})(x^{m}-\lambda_{i}^{-m})^{v} \right)^{r_{i}}.$}
In particular, if $k=p \geq 3$ is a prime integer, then $u=1$ and $v \in \{1,p-1\}$.
\end{rema}

\subsection{The case when $\bar{N}$ contains ${\mathcal A}_{4}$}
We may assume ${\mathcal A}_{4}=\langle a(x)=\frac{i-x}{i+x}, b(x)=-x\rangle$.
Note that ${\rm Fix}(a)=\{\mu_{1},\mu_{2}\}$, where $\mu_{j}$ is solution of $\mu^{2}+(1+i)\mu-i=0$.
In this case, we know that $u=u_{\alpha}$ is such that $u^{3} \equiv 1 \mod(k)$,  
{\small
$$\alpha(x,y)=\left(\frac{i-x}{i+x}, Q_{\alpha}(x) y^{u} \right), \; \beta(x,y)=(-x,Q_{\beta}(x) y),$$
}
where $Q_{\alpha}(x), Q_{\beta}(x)$ are explicitly defined by \eqref{formaQ} and \eqref{formaQ1}, and 
{\small
$$
S: y^{k}=\left(\prod_{j=1}^{2}\left((x^{2}-\mu_{j}^{2})(x-a(-\mu_{j}))(x-a^{2}(-\mu_{j})) \right)^{s_{j}}\right)  \cdot  \left(x\left(x^{2}-1\right)^{u} \left(x^{2}+1 \right)^{u^{2}}   \right)^{s_{3}}\cdot
$$
$$
\cdot \prod_{i=1}^{l} \left(\left(x^{2}-\lambda_{i}^{2}\right) \left(x^{2}-aba^{2}(\lambda_{i})^{2}\right)  \left(x^{2}-a(\lambda_{i})^{2}\right)^{u} 
\left(x^{2}-a(-\lambda_{i})^{2}\right)^{u}  \left(x^{2}-a^{2}(\lambda_{i})^{2}\right)^{u^{2}} \left(x^{2}-a^{2}(-\lambda_{i})^{2}\right)^{u^{2}}  \right)^{r_{i}},
$$
}
where 
\begin{enumerate}[leftmargin=15pt]
\item $s_{j} \in \{0,1,\ldots,k-1\}$, and if $s_{j} \neq 0$, then $\gcd(k,s_{j})=1$;
\item if either $s_{1} \neq 0$ or $s_{2} \neq 0$, then $u=1$;
\item $\lambda_{i} \in {\mathbb C}\setminus \{ 0,\pm 1, \pm i, \pm \mu_{1}, \pm \mu_{2}, a(\pm\mu_{1}), a^{2}(\pm \mu_{1}), a(\pm\mu_{2}), a^{2}(\pm \mu_{2}) \}$;
\item the values $\lambda_{i}$ are in different ${\mathcal A}_{4}$-orbits;
\item if $\delta=4(s_{1}+s_{2})+(1+u+u^{2})\left(s_{3}+4\sum_{i=1}^{l}r_{i}\right)$, then 
\begin{enumerate}
\item if $s_{3} \neq 0$, then $\gcd(k,\delta)=1$; and 
\item of $s_{3}=0$, then $\delta \equiv 0 \mod(k)$.
\end{enumerate}

\end{enumerate}

\begin{rema}
If $u=1$, then the above equations were also provided in \cite[Table 4]{SS}. In the generic situation, $N=\langle \tau, \alpha, \beta \rangle$.  
\end{rema}

\subsection{The case when $\bar{N}$ contains ${\mathfrak S}_{4}$}
We may assume ${\mathfrak S}_{4}=\langle a(x)=ix, b(x)=\frac{1-x}{1+x}\rangle$. Note that ${\rm Fix}(b)=\{\mu_{1}=-1+\sqrt{2},\mu_{2}=-1-\sqrt{2}\}$, and 
${\rm Fix}(ab)=\{\mu_{3},\mu_{4}\}$, where $\mu_{3}$ and $\mu_{4}$ are the solutions of $\mu^{2}+(1-i)\mu+i=0$.
In this case, we know that $u=u_{\alpha}=u_{\beta}$ is such that $u^{2} \equiv 1 \mod(k)$, 
{\small
$$\alpha(x,y)=\left(ix, Q_{\alpha}(x) y^{u} \right), \; \beta(x,y)=\left(\frac{1-x}{1+x}, Q_{\alpha}(x) y^{u} \right),$$
}
where $Q_{\alpha}(x), Q_{\beta}(x)$ are explicitly defined by \eqref{formaQ} and \eqref{formaQ1}, and 
{\small
$$
S: y^{k}=\left(x(x^{4}-1)\right)^{s_{1}} \left((x^{4}-\mu_{1}^{4})(x^{4}-\mu_{2}^{4})(x^{4}-b(i\mu_{1})^{4})  \right)^{s_{2}}
\left( (x^{2}-\mu_{3}^{2}) (x^{2}-\mu_{4}^{2})   (x^{2}+\mu_{3}^{2})^{u} (x^{2}+\mu_{4}^{2})^{u}   \right)^{s_{3}}\cdot
$$
$$
\cdot \prod_{i=1}^{l} \left( 
(x^{2}-\lambda_{i}^{2})(x^{2}+\lambda_{i}^{2})^{u}            
(x^{2}+b(\lambda_{i})^{2})(x^{2}-b(\lambda_{i})^{2})^{u} 
(x^{2}-((ab)^{2}(\lambda_{i}))^{2})(x^{2}+((ab)^{2}(\lambda_{i}))^{2})^{u} \cdot
\right.
$$
$$
\left.
(x^{2}+(aba(\lambda_{i}))^{2})(x^{2}-(aba(\lambda_{i}))^{2})^{u}
(x^{2}+(ab(-aba(\lambda_{i})))^{2})(x^{2}-(ab(-aba(\lambda_{i})))^{2})^{u}
\right.
$$
$$
\left.
(x^{2}+((ab)^{2}(-aba(\lambda_{i})))^{2})(x^{2}-((ab)^{2}(-aba(\lambda_{i})))^{2})^{u}
 \right)^{r_{i}},
$$
}
where 
\begin{enumerate}[leftmargin=15pt]
\item $s_{j} \in \{0,1,\ldots,k-1\}$, and if $s_{j} \neq 0$, then $\gcd(k,s_{j})=1$;
\item if either $s_{1} \neq 0$ or $s_{2} \neq 0$, then $u=1$;
\item $\lambda_{i} \in {\mathbb C}\setminus \{ 0,\pm 1, \pm i, \pm \mu_{j}, \pm i\mu_{j}, \pm b(\pm i\mu_{1}) \}$;
\item the values $\lambda_{i}$ are in different ${\mathfrak S}_{4}$-orbits;
\item if $\delta=5s_{1}+8s_{2}+4(1+u)s_{3}+12(1+u)\sum_{i=1}^{l}r_{i}$, then 
\begin{enumerate}
\item if $s_{1} \neq 0$, then $\gcd(k,\delta)=1$; and 
\item of $s_{1}=0$, then $\delta \equiv 0 \mod(k)$.
\end{enumerate}

\end{enumerate}

\begin{rema}
If $u=1$, then the above equations were also provided in \cite[Table 4]{SS}. In the generic situation, $N=\langle \tau, \alpha, \beta \rangle$.  
\end{rema}

\subsection{The case when $\bar{N}$ contains ${\mathcal A}_{5}$}
Let us now assume $\bar{N}$ contains ${\mathcal A}_{5}$. We may assume 
${\mathcal A}_{5}=\langle a(x)=\omega_{5}x, b(x)=-(x+\lambda_{0})/(\lambda_{0}x-1) \rangle$, where $\lambda_{0}=-i\omega_{5}(1+\omega_{5}^{3})$.
If $\alpha,\beta \in N$ are such that $\theta(\alpha)=a$ and $\theta(\beta)=b$, then (by Lemma \ref{casoA5})
$\alpha(x,y)=(\omega_{5}x, Q_{1}(x)y)$, $\beta(x,y)=(b(x),Q_{2}(x)y)$; so $\tau$ is central in $\langle \tau, \alpha, \beta \rangle$. In this case, an algebraic model for $S$ is as described in \cite[lines $\#$ 24-31 in Table4]{SS},

\section{Example: $n \in \{2,3,4\}$ and $\tau$ tame}\label{Sec:234}
We proceed to explicitly describe the above, when $\bar{N}$ is non-trivial,  $n \in \{2,3,4\}$ and $\tau$ is assumed to be tame (see Table \ref{tabla1}).

\begin{center}
{\small
\begin{table}[h]\label{tabla1}
\begin{tabular}{|c|c|c|||c|c|c|}\hline $n$ & $\bar{N}$ & $S/N$ & $n$ & $\bar{N}$ & $S/N$  \\\hline\hline 
2 & ${\mathbb Z}_{2}$ & $(0;2,k,2k)$, $k \geq 5$ odd  & 
2 & ${\mathbb Z}_{3}$ & $(0;3,3,k)$, $k \geq 5$ odd    \\\hline
2 & ${\mathbb D}_{3}$ &  NO &
3 & ${\mathbb Z}_{2}$ & $(0;2,k,k,2k)$, $k \geq 3$  \\\hline 
3 & ${\mathbb Z}_{3}$ & $(0;3,k,3k)$, $k \geq 3$   &
3 & ${\mathbb Z}_{4}$ & $(0;2,4,2k)$, $k \geq 3$   \\\hline
3 & ${\mathbb Z}_{2}^{2}$ & $(0;2,2,2,4)$, $k=4$  & 
3 & ${\mathbb Z}_{2}^{2}$ & $(0;2,2k,2k)$, $k \geq 3$  \\\hline 
3 & ${\mathbb D}_{4}$ & $(0;2,4,8)$, $k=4$  &
3 & ${\mathcal A}_{4}$ & $(0;2,3,12)$, $k=4$  \\\hline
4 & ${\mathbb Z}_{2}$ & $(0;2,k,k,2k)$, $k \geq 3$ odd  & 
4 & ${\mathbb Z}_{3}$ & $(0;k,3k,3k)$, $k \geq 3$ odd  \\\hline
4 & ${\mathbb Z}_{4}$ & $(0;4,k,4k)$, $k \geq 3$ odd   &
4 & ${\mathbb Z}_{5}$ & $(0;5,5,k)$, $k \geq 3$ odd   \\\hline 
4 & ${\mathbb D}_{3}$ & $(0;2,3,2k)$, $k \equiv 2 \mod(3)$ odd  & 
4 & ${\mathbb D}_{5}$ & $(0;2,5,10)$, $k=5$  \\\hline
\end{tabular}
\caption{The possibilities for $\bar{N}$ non-trivial and signature of $S/N$ for $n=2,3,4$}
\end{table}
}
\end{center}

\subsection{\bf Case $\bf n=2$} In this case, the corresponding Riemann surfaces have genus $(k-1)/2$, so $k \geq 5$ is odd. Moreover, if 
$k=p \geq 11$ is a prime integer, then $N={\rm Aut}(S)$. 

Up to ${\rm PSL}_{2}({\mathbb C})$, we may assume 
${\mathcal B}=\{t_{1}=1, t_{2}=\omega_{3}, t_{3}=\omega_{3}^{2}\}$, whose stabilizer is ${\mathbb D}_{3}=\langle a(x)=\omega_{3}x, b(x)=1/x\rangle$. 
So, the possibilities for non-trivial $\bar{N}$ are ${\mathbb D}_{3}$, ${\mathbb Z}_{2}=\langle b \rangle$ (up to conjugation), and ${\mathbb Z}_{3}=\langle a \rangle$.
As already seen in Example \ref{maxdihe}, the case ${\mathbb D}_{3}$ is not possible.

\subsubsection{}
 Let $\bar{N}={\mathbb Z}_{2}=\langle b \rangle$, and let $\beta \in N$ be such that $\theta(\beta)=b$. 
In this case,
{\small 
$$S:y^{k}=(x-1)(x-\omega_{3})^{\frac{k-1}{2}}(x-\omega_{3}^{2})^{\frac{k-1}{2}}, \; 
\beta(x,y)=\left( \frac{1}{x},\frac{-y}{x}  \right),$$ 
$$N=\langle \tau, \beta: \tau^{k}=\beta^{2}=\beta\tau\beta\tau^{-1}=1 \rangle \cong {\mathbb Z}_{k} \times {\mathbb Z}_{2}.$$
}

\subsubsection{} Let $\bar{N}={\mathbb Z}_{3}=\langle a \rangle$ and let $\alpha \in N$ be such that $\theta(\alpha)=a$. 
In this case, there is some
$u \in \{1,\ldots,k-1\}$, where $1+u+u^{2} \equiv 0 \mod(k)$, such that
{\small
$$S: y^{k}=(x-1)(x-\omega_{3})^{u}(x-\omega_{3}^{2})^{u^{2}}, \; 
\alpha(x,y)=\left( \omega_{3}x,\omega_{3}^{\frac{1+u+u^{2}}{k}}y^{u} \right),$$
$$N=\langle \tau \alpha: \tau^{k}=\alpha^{3}=\alpha\tau\alpha^{-1}\tau^{-u}=1\rangle \cong {\mathbb Z}_{k} \rtimes {\mathbb Z}_{3}.$$
}

Note that, for $k=5$ the above is not possible as there is no $u \in \{1,\ldots,4\}$ such that $1+u+u^{2}$ is divisible by $5$.  For $k=7$, this works with $u=2$.

\begin{rema}
 For triples $(l_{1},l_{2},l_{3})$ such that $l_{j}\in\{1,\ldots,k-1\}$, $\gcd(k,l_{j})=1$, $l_{1}+l_{2}+l_{3} \equiv 0 \mod(k)$, such that not two coordinates are equal and there is no $u \in \{1,\ldots,k-1\}$, $1+u+u^{2} \equiv 0 \mod(k)$ such that $\{l_{1},l_{2},l_{3}\}=\{1,u,u^{2}\}$ (up to ${\mathbb Z}_{k}^{*}$), and {\small $S: y^{k}=(x-1)^{l_{1}}(x-\omega_{3})^{l_{2}}(x-\omega_{3}^{2})^{l_{3}}$} will satisfies that $N=\langle \tau \rangle$.
\end{rema}

\subsection{\bf Case $\bf n=3$} In this case, the corresponding Riemann surfaces have genus $k-1$, so $k \geq 3$.  Moreover, if $k=p \geq 13$ is a prime integer, then $N={\rm Aut}(S)$.
The only possibilities for  $\bar{N}$  are ${\mathbb Z}_{2}$, ${\mathbb Z}_{3}$, ${\mathbb Z}_{4}$, ${\mathbb Z}_{2}^{2}$, ${\mathbb D}_{4}$ and ${\mathcal A}_{4}$.

\subsubsection{} $\bar{N}={\mathcal A}_{4}$ is only possible for $k=4$ and the genus three curve {\small $S: y^{4}=x^{3}-1$}.

\subsubsection{} $\bar{N}={\mathbb D}_{4}$ is only possible for $k=4$ and the genus three curve {\small $S: y^{4}=x^{4}-1$}.

\subsubsection{} $\bar{N}={\mathbb Z}_{2}^{2}$ and $S/N$ of signature $(0;2,2k,2k)$ is produced by the genus $k-1$ curves
{\small
$$S: y^{k}=(x^{2}-1)(x^{2}+1)^{\frac{k-2}{2}}, \quad \mbox{$k \geq 4$ even; or}$$
$$S: y^{k}=(x^{2}-1)(x^{2}+1)^{k-1}, \quad \mbox{for any $k \geq 3$}.$$
}
In these cases, $N=\langle \tau, \alpha, \beta: \tau^{k}=\alpha^{2}=1, \beta^{2}=\tau, (\alpha\beta)^{2}=\tau^{r},\alpha \tau \alpha=\tau, \beta \tau \beta^{-1} = \tau\rangle$,  where
$r=(k+2)/2$ in the first case, and $r=1$ in the second case, and 
{\small
$$\alpha(x,y)=(-x,y), \; 
\beta(x,y)=\left\{ \begin{array}{ll} 
\left(\frac{1}{x},\frac{\omega_{2k}y}{x}\right), & \mbox{in the first case $k \geq 4$ even}\\
\left(\frac{1}{x},\frac{\omega_{2k}y}{x^{2}}\right), & \mbox{in the second case}.
\end{array}
\right.
$$
}

\subsubsection{} $\bar{N}={\mathbb Z}_{2}^{2}$ and $S/N$ of signature $(0;2,2,2,4)$, produced by the genus three curves
{\small
$$S: y^{4}=(x^{2}-\lambda^{2})(x^{2}-\lambda^{-2}), \quad \mbox{$\lambda \in {\mathbb C} \setminus \{0,\pm 1, \pm i\}$}.$$
}
In this case, $N=\langle \tau, \alpha, \beta: \tau^{4}=\alpha^{2}=\beta^{2}=1, (\alpha\beta)^{2}=\tau^{2}, \alpha \tau \alpha=\tau, \beta \tau \beta = \tau\rangle$, where
{\small
$$\alpha(x,y)=(-x,y), \; \beta(x,y)=\left(1/x,y/x\right).$$}

\subsubsection{} $\bar{N}={\mathbb Z}_{4}$. This situation occurs if there exists some $u \in \{1,\ldots,k-1\}$, $u^{4} \equiv 1 \mod(k)$, $1+u+u^{2}+u^{3} \equiv 0 \mod(k)$, and the curve is given by
{\small
$$S: y^{k}=(x-1)(x-i)^{u}(x+1)^{u^{2}}(x+i)^{u^{3}},$$ which admits the automorphism
$$\alpha(x,y)=\left(ix,\frac{i^{\frac{1+u+u^{2}+u^{3}}{k}} y^{u}}{(x+i)^{\frac{u^{4}-1}{k}}} \right),$$
}

\subsubsection{} $\bar{N}={\mathbb Z}_{3}$. This occurs if there is some $l \in \{1,\ldots,k-1\}$, $\gcd(k,l)=1$, $1+3l \equiv 0 \mod(k)$, and the curve is given by
{\small $$S: y^{k}=x(x^{3}-1)^{l},$$}
which admits the automorphism
{\small $\alpha(x,y)=(\omega_{3}x, \omega_{3}^{\frac{1+3l}{k}} y).$}

\subsubsection{} $\bar{N}={\mathbb Z}_{2}$ and $S/N$ of signature $(0;2,2,k,k)$. This occurs if there are $u, l \in \{1,\ldots,k-1\}$, $\gcd(k,l)=1$, $u^{2} \equiv 1 \mod(k)$, $(1+u)(l+1) \equiv 0 \mod(k)$, and the curve is given by
{\small
$$S: y^{k}=(x-1)(x+1)^{u}(x-i)^{l}(x+i)^{ul}$$
}
which admits the automorphism
{\small
$$\alpha(x,y)=\left(-x, \frac{(-1)^{\frac{(u+1)(l+1)}{k}}y^{u}}{(x+1)^{\frac{u^{2}-1}{k}} (x+i)^{\frac{l(u^{2}-1)}{k}}  }\right).$$
}

\subsubsection{} $\bar{N}={\mathbb Z}_{2}$ and $S/N$ of signature $(0;k,2k,2k)$. This occurs for the curve
{\small
$$S: y^{k}=x(x^{2}-1)^{l},$$
}
 where $l \in \{1,\ldots,k-1\}$, $\gcd(k,l)=1$, $\gcd(k,1+2l)=1$, which admits the automorphism
 {\small
 $
 \alpha(x,y)=(-x, (-1)^{k} y  ).
 $
}

\subsection{\bf Case $\bf n=4$} In this case, the corresponding Riemann surfaces have genus $3(k-1)/2$, so $k \geq 3$ is odd. Moreover, if 
$k=p \geq 19$ is a prime integer, then $N={\rm Aut}(S)$.
The only possibilities for $\bar{N}$  are ${\mathbb Z}_{2}$, ${\mathbb Z}_{3}$, ${\mathbb Z}_{4}$, ${\mathbb Z}_{5}$, ${\mathbb D}_{3}$ and ${\mathbb D}_{5}$.

\subsubsection{} $\bar{N}={\mathbb D}_{5}$ is only possible for $k=5$ and the genus six curve {\small $S: y^{5}=x^{5}-1$}.

\subsubsection{} $\bar{N}={\mathbb D}_{3}$ only happens for $k \equiv 2 \mod(3)$ and {\small $S: y^{k}=x(x^{3}-1)^{\frac{k-2}{3}}$}.

\subsubsection{} $\bar{N}={\mathbb Z}_{5}$ happens for those $k \geq 3$ for which there exists some $u \in \{1,\ldots,k-1\}$, such that $\gcd(k,u)=1$ and $1+u+u^{2}+u^{3}+u^{4} \equiv 0 \mod(k)$. In this case, 
{\small
$$S: y^{k}=(x-1)(x-\omega_{5})^{u}(x-\omega_{5}^{2})^{u^{2}}(x-\omega_{5}^{3})^{u^{3}}(x-\omega_{5}^{4})^{u^{4}},$$
} which admits the automorphism
{\small
$$\alpha(x,y)=\left(\omega_{5}x, \omega_{5}^{\frac{1+u+u^{2}+u^{3}+u^{4}}{k}} \frac{y^{u}}{(x-\omega_{5}^{4})^{\frac{u^{5}-1}{k}}}\right).$$
}

\subsubsection{} $\bar{N}={\mathbb Z}_{3}$. Working as in previous cases, by taking $a(x)=\omega_{3}x \in \bar{N}$, one obtain that  
{\small
$$S: y^{k}=x(x^{3}-1)^{l},$$
} where $l \in \{1,\ldots,k-1\}$, $\gcd(k,l)=1$ and $\gcd(k,1+3l)=1$.
This curve admits the automorphisms 
{\small
$$\alpha(x,y)=\left(\omega_{3}x, \omega_{3}^{\frac{1+3l}{k}}y\right)$$
}

\begin{rema}
Let us observe that, in the case that  $2+3l \equiv 0 \mod(k)$, the curve also admits the automorphism $\beta(x,y)=\left(\frac{1}{x}, \frac{(-1)^{1/k}y}{x^{\frac{2+3l}{k}}} \right) \in N,$
in which  case, $\bar{N}=\langle a(x)=\omega_{3}x, b(x)=1/x\rangle \cong {\mathbb D}_{3}$.
\end{rema}

\subsubsection{} $\bar{N}={\mathbb Z}_{2}$. In this case, we may assume ${\mathbb Z}_{2}=\langle a(x)=-x \rangle$ and 
{\small
$$S: y^{k}=x(x^{2}-1)^{l_{1}}(x^{2}-\lambda^{2})^{l_{2}},$$
}
where $\lambda^{2} \neq 0,1$, and $l_{1},l_{2} \in \{1,\ldots,k-1\}$, $\gcd(k,l_{1})=1=\gcd(k,l_{2})$, $1+2(l_{1}+l_{2}) \equiv 0 \mod(k)$.
This curve admits the automorphism $\alpha(x,y)=(-x, \omega_{2k} y).$

\section{Example: $n=5$ and $\tau$ tame}\label{Sec:5}
In this section, we assume $n=5$ and $\tau$ is a tame $k$-gonal automorphism.
So, $g=2(k-1)$ and $k\geq 2$. Moreover, if either $k=2$ or $k=p \geq 23$ is a prime integer, then $N={\rm Aut}(S)$.
In this case, 
${\mathcal B}=\{t_{1},\ldots,t_{6}\}$, and 
{\small
$$S: y^{k}=(x-t_{1})^{l_{1}}(x-t_{2})^{l_{2}}(x-t_{3})^{l_{3}} (x-t_{4})^{l_{4}} (x-t_{5})^{l_{5}} (x-t_{6})^{l_{6}},$$
}
for suitable values of $l_{j} \in \{1,\ldots,k-1\}$ such that $\gcd(k,l_{j})=1$ and $l_{1}+\cdots+l_{6} \equiv 0 \mod(k)$. We may always assume $l_{1}=1$.

The only possibilities for non-trivial $\bar{N}$ are 
the cyclic groups of orders $2, 3, 4, 5, 6$, or the dihedral groups ${\mathbb D}_{2}={\mathbb Z}_{2}^{2}$, ${\mathbb D}_{3}, {\mathbb D}_{4}, {\mathbb D}_{6}$, or ${\mathcal A}_{4}$ or ${\mathfrak S}_{4}$. For these cases, respectively, $|{\rm Aut}(S)| \geq |N|=4k, 6k,8k,12k, 12k,24k$.

\begin{rema}
If $\gcd(k,30)=1$, then $N=\langle \tau \rangle \rtimes \bar{N}$ (by the Schur-Zassenhaus theorem)
\end{rema}

\begin{rema}
If $\bar{N} \in \{{\mathbb Z}_{4}, {\mathbb Z}_{5}, {\mathbb Z}_{6}, {\mathbb D}_{4}, {\mathbb D}_{6}, {\mathcal A}_{4}, {\mathfrak S}_{4}\}$, then $S/N$ is a triangular orbifold. 
The only cases for which the genus zero quotient orbifold $S/N$ may have four cone points (so these Riemann surfaces form a one-dimensional family) is when $\bar{N} \in \{ {\mathbb Z}_{2}^{2}, {\mathbb D}_{3}\}$. 
\end{rema}

\subsection{} Assume $\bar{N}={\mathfrak S}_{4}$. We may assume ${\mathfrak S}_{4}=\langle a(x)=ix, b(x)=(1-x)/(1+x)\rangle$ and $t_{1}=0, t_{2}=1, t_{3}=i, t_{4}=-1, t_{5}=-i, t_{6}=\infty$.
Let $\alpha, \beta \in N$ be such that $\theta(\alpha)=a$ and $\theta(\beta)=b$. Since $a(t_{1})=t_{1}$, it follows that $u_{\alpha}=1$. Similarly, as $bab(t_{2})=t_{2}$, it also follows that $u_{\beta\alpha\beta^{-1}}=1$. 
By following the action of $a$, we obtain that $l_{2}=l_{3}=l_{4}=l_{5}$ and, by following the action of $bab$ we obtain that $l_{1}=l_{2}$. In particular, as we may assume $l_{1}=1$, we obtain that $6 \equiv 0 \mod(k)$, that is, $k \in \{2,3,6\}$, and 
{\small
$$S: y^{k}=x(x^{4}-1), \quad 
\alpha(x,y)=(ix, \omega_{4k}y),  \; \beta(x,y)=\left(\frac{1-x}{1+x}, \frac{8^{1/k}y}{(x+1)^{6/k}}\right).$$
}

The signature of $S/N$ is $(0;2,3,4k)$.

\subsection{} Assume $\bar{N}$ contains ${\mathcal A}_{4}$.
We may assume ${\mathcal A}_{4}=\langle a(x)=-x, b(x)=(i-x)/(i+x)\rangle$ and $t_{1}=0, t_{2}=1, t_{3}=i, t_{4}=-1, t_{5}=-i, t_{6}=\infty$.
Let $\alpha, \beta \in N$ be such that $\theta(\alpha)=a$ and $\theta(\beta)=b$.
Since $a(t_{1})=t_{1}$, $u_{\alpha}=1$. Let $u=u_{\beta}$, which must satisfy that $u^{3} \equiv 1 \mod(k)$.
We may assume $l_{1}=1$. By following the action of $b$, we obtain (equalities are modulo $(k)$) that $l_{2}=u$, $l_{3}=u^{2}$, $l_{5}=ul_{4}$ and $l_{6}=u^{2}l_{4}$. By following the action of $a$ we obtain that $l_{4}=l_{2}$ and $l_{5}=l_{3}$. In particular, $l_{2}=l_{4}=u$, $l_{3}=l_{5}=u^{2}$, $l_{6}=1$, and $2(1+u+u^{2}) \equiv 0 \mod(k)$. It follows that 
{\small
$$S: y^{k}=x(x^{2}-1)^{u}(x^{2}+1)^{u^{2}},$$
$$\alpha(x,y)=(-x, \omega_{2k}y), \; \beta(x,y)=\left(\frac{i-x}{i+x}, \frac{(-4i)^{\frac{u}{k}} 2^{\frac{u^{2}}{k}} (-1)^{ \frac{1+u^{2}}{k}} y^{u}}{(x-i)^{\frac{u^{3}-1}{k}} (x+i)^{\frac{u^{3}-1}{k}+\frac{2(1+u+u^{2})}{k}}} \right).$$
}

The signature of $S/N$ is $(0;3,3,2k)$.

\begin{rema}
Note that, if $u=1$, then we obtain the case when $\bar{N}={\mathfrak S}_{4}$.

\end{rema}

\subsection{} Assume $\bar{N}$ contains ${\mathbb D}_{6}$.
We may assume ${\mathbb D}_{6}=\langle a(x)=\omega_{6}x, b(x)=1/x\rangle$ and $t_{1}=1, t_{2}=\omega_{6}, t_{3}=\omega_{6}^{2}=\omega_{3}, t_{4}=\omega_{6}^{3}=-1, t_{5}=\omega_{6}^{4}=-\omega_{6}, t_{6}=\omega_{6}^{5}=-\omega_{3}$.
Let $\alpha, \beta \in N$ be such that $\theta(\alpha)=a$ and $\theta(\beta)=b$.
Since $b(t_{1})=t_{1}$, $u_{\beta}=1$. Let $u=u_{\alpha}$, which must satisfy that $u^{6} \equiv 1 \mod(k)$.
We may assume $l_{1}=1$. So, by following $a$, we obtain that $l_{2}=u$, $l_{3}=u^{2}$, $l_{4}=u^{3}$, $l_{5}=u^{4}$ and $l_{6}=u^{5}$. By following $b$, we observe that $l_{6}=l_{2}$. In particular, $u^{2} \equiv 1 \mod(k)$. The equation $l_{1}+\cdots+l_{6} \equiv 0 \mod(k)$ now asserts that $3(1+u) \equiv 0 \mod(k)$, and 
{\small
$$S: y^{k}=(x^{3}-1)(x^{3}+1)^{u},$$
$$\alpha(x,y)=\left(\omega_{6}x, \frac{\omega_{2k}^{1+u} y^{u}}{(x^{3}+1)^{\frac{u^{2}-1}{k}}}  \right), \; \beta(x,y)=\left(\frac{1}{x},\frac{\omega_{2k} y}{x^{\frac{3(1+u)}{k}} } \right).$$
}

The signature of $S/N$ is $(0;2,6,2k)$.

\begin{rema}
If $k \neq 3$, then the above signature is finitely maximal \cite{Singerman}, so ${\rm Aut}(S)=G=N$ and $\bar{N} \cong {\mathbb D}_{6}$. This surface, for the case $k=p \geq 5$ a prime integer, also appeared in \cite{CR}. In that paper, the authors studied those cyclic $p$-gonal curves of genus $2(p-1)$ admitting a group of automorphisms of order a multiple of $4p$. Here, we are considering those whose groups of automorphisms have order a multiple of $6p$.
\end{rema}

\begin{rema}
In the case that $\gcd(k,3)=1$, then $u=k-1$. Moreover, 
$\alpha \tau \alpha^{-1}=\tau^{-1}$, $\beta \tau \beta^{-1}=\tau$ and 
$\alpha$ has order $6$.

(1) If $k \geq 5$ is odd, then both $\beta$ and $\alpha\beta$ have order $2$. So, in this case,
$\langle \alpha, \beta \rangle \cong {\mathbb D}_{6},$
and 
$$G=\langle \tau, \alpha, \beta: \tau^{k}=\alpha^{6}=\beta^{2}=(\alpha\beta)^{2}=1, \alpha \tau \alpha^{-1}=\tau^{-1},  \beta \tau \beta=\tau
 \rangle =\langle \tau \rangle \rtimes \langle \alpha,\beta \rangle \cong {\mathbb Z}_{k} \rtimes {\mathbb D}_{6}.$$

(2) If $k \geq 4$ is even, then $\beta^{2}=\tau$ and $\alpha\beta$ has order $2$. So, in this case, 
 $$G=\langle \tau, \alpha, \beta: \tau^{k}=\alpha^{6}=(\alpha\beta)^{2}=1, \beta^{2}=\tau,\alpha \tau \alpha^{-1}=\tau^{-1},  \beta \tau \beta^{-1}=\tau
 \rangle.$$ 
 
 \end{rema}

\subsection{} Assume $\bar{N}$ contains ${\mathbb D}_{4}$.
We may assume ${\mathbb D}_{4}=\langle a(x)=ix, b(x)=1/x\rangle$ and $t_{1}=0, t_{2}=1, t_{3}=i, t_{4}=-1, t_{5}=-i, t_{6}=\infty$.
Let $\alpha, \beta \in N$ be such that $\theta(\alpha)=a$ and $\theta(\beta)=b$. As $a$ and $b$ both fix points in ${\mathcal B}$, it follows that $u_{\alpha}=u_{\beta}=1$. Following the action of $a$, we obtain that $l_{2}=l_{3}=l_{4}=l_{5}$, and by following the action of $b$, $l_{6}=l_{1}$. As we may assume $l_{1}=1$, 
{\small
$$S: y^{k}=x(x^{4}-1)^{l},$$
}
where $l \in \{1,\ldots,k-1\}$, $\gcd(k,l)=1$, and $2(1+2l) \equiv 0 \mod(k)$, and 
{\small
$$\alpha(x,y)= \left(ix,\omega_{4k}y\right), \; \beta(x,y)=\left(\frac{1}{x}, \frac{\omega_{2k}^{l} y}{x^{\frac{2(1+2l)}{k}}} \right).$$
}

The signature of $S/N$ is $(0;2,2k,4k)$. Note that, if $\gcd(k,2)=1$, then $l=\frac{k-1}{2}$.

\subsection{}
Assume $\bar{N}$ contains ${\mathbb Z}_{2}^{2}=\langle a,b: a^{2}=b^{2}=(ab)^{2}=1\rangle$. The case when $k \geq 5$ is a prime integer has been studied in \cite{CR}.
Up to ${\rm PSL}_{2}({\mathbb C})$, we may assume $a(x)=-x$, $b(x)=1/x$.
There are two cases we must consider: 
\begin{enumerate}[leftmargin=15pt]
\item[(i)] ${\mathcal B}$ consists of three ${\mathbb Z}_{2}^{2}$-orbits (corresponding to the fixed points of all elements of order two). 
\item[(ii)] ${\mathcal B}$ consists of two ${\mathbb Z}_{2}^{2}$-orbits (one is the locus of fixed points of one of the involutions).
\end{enumerate}

\subsubsection{\rm Case (i)}
In this case, $S/\theta^{-1}({\mathbb Z}_{2}^{2})$ has signature $(0;2k,2k,2k)$ and 
$t_{1}=0, t_{2}=1, t_{3}=i, t_{4}=-1, t_{5}=-i, t_{6}=\infty$.
Let $\alpha, \beta \in N$ be such that $\theta(\alpha)=a$ and $\theta(\beta)=b$. As each $a, b$, and $ab$ has a fixed point in ${\mathcal B}$, then $u_{\alpha}=u_{\beta}=u_{\alpha\beta}=1$.
As we may assume $l_{1}=1$, by following the actions of $a$ and $b$, we obtain
$l_{4}=l_{2}, l_{5}=l_{3}, l_{6}=l_{1}=1.$
Since, $1+\cdots+l_{6} \equiv 0 \mod(k)$,  we must have $2(1+l_{2}+l_{3}) \equiv 0 \mod(k)$, and 
{\small
$$S: y^{k}=x(x^{2}-1)^{l_{2}}(x^{2}+1)^{l_{3}},$$
$$\alpha(x,y)=(-x,\omega_{2k}y), \; 
\beta(x,y)=\left(\frac{1}{x}, \frac{\omega_{2k}^{l_{2}} y}{x^{2(1+l_{2}+l_{3})/k} }  \right).$$
}

We note that $\alpha^{2}=\tau$, $\beta^{2}=\tau^{l_{2}}$ and $(\alpha\beta)^{2}=\tau^{l_{2}+1}$. Also, if $k$ odd, then $1+l_{2}+l_{3} \equiv 0 \mod(k)$.

\begin{rema}
The above curves depend on a choice of the triple $(l_{1},l_{2},l_{3})$, where $l_{j} \in \{1,\ldots,k-1\}$, $\gcd(k,l_{j})=1$ and $2(l_{1}+l_{2}+l_{3}) \equiv 0 \mod(k)$. Two such triplets
$(l_{1},l_{2},l_{3})$ and $(\hat{l}_{1},\hat{l}_{2},\hat{l}_{3})$ define isomorphic curves if and only if there is some $u  \in \{1,\ldots,k-1\}$, $\gcd(k,u)=1$, and a permutation $\sigma \in {\mathfrak S}_{3}$, such that $\hat{l}_{j}=u l_{\sigma(j)}$, for $j=1,2,3$. Some examples are the following.

\begin{enumerate}[leftmargin=15pt]
\item If $k=5$, then $g=8$, and (by taking $l_{1}=1$) we have the following cases: $$(l_{1},l_{2},l_{3}) \in \{(1,2,2), (1,1,3), (1,3,1)\}.$$ All of them are isomorphic.
For instance, the first possibility produces the curve:
{\small $S: y^{5}=x(x^{2}-1)^{2}(x^{2}+1)^{2}=x(x^{4}-1)^{2}.$}

\item If $k=7$, then  $g=12$, and (by taking $l_{1}=1$) we have the following cases: $$(l_{1},l_{2},l_{3}) \in \{(1,3,3), (1,2,4), (1,1,5), (1,4,2)\}.$$ The triplets $(1,3,3)$, $(1,1,5)$ produce isomorphic curves; the same for the triplets $(1,2,4)$ and $(1,4,2)$. But the triples $(1,1,5)$ and $(1,2,4)$ are non-isomorphic. These two curves are
{\small $S_{1}: y^{7}=x(x^{2}-1)(x^{2}+1)^{5},$} and {\small $S_{2}: y^{7}=x(x^{2}-1)^{2}(x^{2}+1)^{4}.$}

\item If there exists $r \in \{1,\ldots,k-1\}$ such that $r^{3} \equiv 1 \mod (k)$ and $1+r+r^{2} \equiv 0 \mod(k)$, then by taking $(l_{1},l_{2},l_{3})=(1,r,r^{2})$, we obtain the curve
{\small $S: y^{k}=x(x^{2}-1)^{r}(x^{2}+1)^{r^{2}},$}
for which $\bar{N}={\mathcal A}_{4}$. For instance, if $k=7$, then $r=2$ and the curve $y^{7}=x(x^{2}-1)^{2}(x^{2}+1)^{4}$ admits the extra automorphism $\delta \in N$ (which satisfies $\delta \tau \delta^{-1}=\tau^{2}$) defined by
{\small
$$\delta(x,y)=\left( \frac{x+i}{x-i}, -\frac{16^{2/7} y^{2}}{(x-i)(x^{2}+1)}\right).$$
}

\end{enumerate}
\end{rema}

\subsubsection{\rm Case (ii)}
We may assume that $t_{1}=0$, $t_{2}=\lambda$, $t_{3}=-1/\lambda$, $t_{4}=-\lambda$, $t_{5}=1/\lambda$ and $t_{6}=\infty$, where $\lambda \in {\mathbb C} \setminus \{0,\pm1, \pm i\}$.
If $\alpha \in N$ is such that $\theta(\alpha)=a$, then (as $a$ fixes a point in ${\mathcal B}$), $u_{\alpha}=1$, i.e., $\alpha(x,y)=(-x,Q_{1}(x) y)$.
If $\beta \in N$ is such that $\theta(\beta)=b$, then $\beta(x,y)=(1/x, Q_{2}(x) y^{l})$, for $t=u_{\beta}$ and $t^{2} \equiv 1 \mod(k)$.
Following the action of $a$ ensures that 
$l_{4}=l_{2}, l_{5}=l_{3}.$
Similarly, by following the action of $b$, we obtain
$l_{5}=t l_{2}, l_{4}=t l_{3}, l_{6}=l_{1}=1.$
In particular, $l_{3}=tl_{2}$ (modulo $k$) and, 
since $l_{1}+\cdots+l_{6} \equiv 0 \mod(k)$, it holds that 
$2(1+(1+t)l_{2}) \equiv 0 \mod(k),$
and 
{\small
$$S: y^{k}=x(x^{2}-\lambda^{2})^{l_{2}}(x^{2}-\lambda^{-2})^{tl_{2}}, \; 
\alpha(x,y)=(-x,\omega_{2k} y).$$
}

So $\alpha^{2}=\tau$. Since
{\small 
$$b(x)(b(x)^{2}-\lambda^{2})^{l_{2}}(b(x)^{2}-\frac{1}{\lambda^{2}})^{tl_{2}}=
(-1)^{l_{2}(1+t)} \lambda^{2l_{2}(1-t)} \frac{y^{k}}{x^{2(1+(1-t)l_{2}-l_{1}(t-1)} x^{t-1} (x^{2}-\lambda^{2})^{l_{2}(t-1)}(x^{2}-\frac{1}{\lambda^{2}})^{tl_{2}(1-t)} }, $$
}
it follows that $t=1$. So,
{\small
$$S: y^{k}=x\left(x^{2}-\lambda^{2}\right)^{l_{2}}\left(x^{2}-\lambda^{-2}\right)^{l_{2}},$$
}
where $2(1+2l_{2}) \equiv 0 \mod(k)$, and 
{\small
$$\beta(x,y)=\left( \frac{1}{x}, \frac{\omega_{k}^{l_{2}}y}{x^{2(1+2l_{2})/k}} \right).$$
}
So, $\beta^{2}=\tau^{2l_{2}}$, and
{\small
$G=\langle \tau, \alpha, \beta: \tau^{k}=1, \alpha^{2}=\tau, \beta^{2}=\tau^{-1}, (\alpha\beta)^{2}=1\rangle \leq N.$
}

\s

The quotient orbifold $S/G$ has signature $(0;2,2,k,2k)$. As this signature is finitely maximal \cite{Singerman}, for the generic value of $\lambda$, we have $G={\rm Aut}(S)$.

\begin{rema}
If $k \geq 3$ is odd, then $l_{2}=\frac{k-1}{2}$ and 
{\small $S: y^{k}=x\left(x^{2}-\lambda^{2}\right)^{\frac{k-1}{2}}\left(x^{2}-\lambda^{-2}\right)^{\frac{k-1}{2}},$}
$\alpha(x,y)=(-x,\omega_{2k}y),$ and 
$\beta(x,y)=\left( \frac{1}{x}, \frac{\omega_{k}^{\frac{k-1}{2}}y}{x^{2}} \right).$
If, moreover,  $\lambda \in \{\pm \omega_{8}, \pm i \omega_{8}\}$, then ${\rm Aut}(S)/\langle \tau \rangle={\mathcal A}_{4}$.

\end{rema}
\subsection{}
Let us assume $\bar{N}$ contains ${\mathbb D}_{3}=\langle a, b: a^{3}=b^{2}=(ab)^{2}=1\rangle$, in particular, that $|{\rm Aut}(S)| \geq 6k$, and 
$\bar{N} \in \{ {\mathbb D}_{3}, {\mathbb D}_{6}, {\mathfrak S}_{4}\}$. Hereafter, to simplify our computations, we will assume $k \geq 4$ such that $\gcd(k,3)=1$.

\begin{rema}
By the Schur-Zassenhaus theorem, if $\gcd(k,6)=1$, then $\theta^{-1}({\mathbb D}_{3})=\langle \tau \rangle \rtimes {\mathbb D}_{3}$.
\end{rema}

There are two cases we must consider: 
\begin{enumerate}[leftmargin=15pt]
\item[(i)] ${\mathcal B}$ consists of two ${\mathbb D}_{3}$-orbits (corresponding to the fixed points of all elements of order two in ${\mathbb D}_{3}$). 
\item[(ii)] ${\mathcal B}$ consists of one ${\mathbb D}_{3}$-orbit.
\end{enumerate}

\subsubsection{\rm Case (i)}\label{caso2}
In this case, $S/\theta^{-1}({\mathbb D}_{3})$ has signature $(0;3,2k,2k)$.
Up to conjugation in ${\rm PSL}_{2}({\mathbb C})$, we may assume 
${\mathbb D}_{3}=\langle a(x)=\omega_{3}x, b(x)=1/x\rangle$ and 
$t_{1}=1$, $t_{2}=\omega_{6}$,  $t_{3}=\omega_{6}^{2}=\omega_{3}$,  $t_{4}=\omega_{6}^{3}=-1$, $t_{5}=\omega_{6}^{4}$, $t_{6}=\omega_{6}^{5}$. 
Let $\alpha, \beta \in N$ be such that $\theta(\alpha)=a$ and $\theta(\beta)=b$. In this case $u=u_{\alpha}$ satisfies that $u^{3} \equiv 1 \mod(k)$, and (as $b$ has fixed points in ${\mathcal B}$) $u_{\beta}=1$.
As we may assume $l_{1}=1$, by following the actions of $a$ and $b$, we must have (equalities are assumed modulo $k$)
$l_{3}=u, l_{5}=u^{2}, l_{4}=ul_{2}, l_{6}=u^{2}l_{2}, l_{6}=l_{2}.$
As $u^{2}l_{2}=l_{6}=l_{2}$ (equalities are modulo $k$)  and $\gcd(k,l_{2})=1$, we obtain that $u^{2} \equiv 1 \mod(k)$. So, $u=1$. Now, as
$0 \equiv l_{1}+\cdots+l_{6} \equiv 3(1+l_{2}) \mod(k)$ and we are assuming $\gcd(k,3)=1$, we obtain that $1+l_{2} \equiv 0 \mod(k)$, that is, $l_{2}=k-1$. So,
{\small $S: y^{k}=(x^{3}-1)(x^{3}+1)^{k-1}.$}
But this case corresponds to have ${\mathbb D}_{6}$ contained in $\bar{N}$.

\subsubsection{\rm Case (ii)}\label{caso1}
Let us now consider the case when ${\mathcal B}$ is a complete ${\mathbb D}_{3}$-orbit. In this case, $S/\theta^{-1}({\mathbb D}_{3})$ has signature $(0;2,2,3,k)$.
We may assume that
$t_{1}=1, \; t_{2}=\omega_{3}, \; t_{3}=\omega_{3}^{2}, \; t_{4}=\lambda_{1}, \; t_{5}=\lambda_{2}, \; t_{6}=\lambda_{3},$
$a(t_{1})=t_{2}, \; a(t_{2})=t_{3}, \; a(t_{3})=t_{1}, \;   a(t_{4})=t_{5}, \; a(t_{5})=t_{6}, \; a_(t_{6})=t_{4}, $
$b(t_{1})=t_{4}, \; a(t_{2})=t_{6}, \; a_(t_{3})=t_{5}.$
In this case, 
$a(x)=\omega_{3} x, \; b(x)=\frac{\lambda_{1}}{x},$
$\lambda_{2}=\omega_{3} \lambda_{1}, \; \lambda_{3}=\omega_{3}^{2} \lambda_{1}.$
The parameter $\lambda_{1}$ belongs to the region $\Omega=\{\lambda \in {\mathbb C}: \lambda^{3} \notin \{0,1\}\}$.
So, in this case, 
{\small
$$S: y^{p}=(x-1)^{l_{1}}(x-\omega_{3})^{l_{2}}(x-\omega_{3}^{2})^{l_{3}} (x-\lambda_{1})^{l_{4}} (x-\omega_{3}\lambda_{1})^{l_{5}} (x-\omega_{3}^{2})^{l_{6}},$$
}for suitable values of $l_{j} \in \{1,\ldots,k-1\}$ such that $\gcd(k,l_{j})=1$ and $l_{1}+\cdots+l_{6} \equiv 0 \mod(k)$, where we may assume $l_{1}=1$.

Let $\alpha, \beta \in N$ be such that $\theta(\alpha)=a$, $\theta(\beta)=b$, and 
$u=u_{\alpha},v=u_{\beta} \in \{1,\ldots,k-1\}$ (so, $\gcd(k,u)=\gcd(k,v)=1$, $u^{3} \equiv 1 \mod(k)$ and $v^{2} \equiv 1 \mod(k)$).
Following the actions of $a$ and $b$, we obtain (the equalities are modulo $k$)
$l_{2}=ul_{1}, \; l_{3}=u^{2}l_{1}, \; l_{1}=u^{3}l_{1},   l_{5}=ul_{4}, \; l_{6}=u^{2}l_{4}, \; l_{4}=u^{3}l_{4}, $
$l_{4}=v l_{1}, \; l_{1}=v^{2}l_{1}, \; l_{5}=v l_{3}, \; l_{6}=v l_{2}.$
Using the equality
$u^{2}l_{1}=l_{3}=v l_{5}= vul_{4}=v^{2}ul_{1}$ (modulo $k$), we obtain that $u \equiv v^{2} \equiv 1 \mod(k)$, i.e, $u=1$. 
Then, $0 \equiv l_{1}+\cdots+l_{6} \equiv 3l_{1}(1+v) \mod(k)$, i.e., (as $\gcd(k,l_{1})=\gcd(k,3)=1$), $v=k-1$.
As we may assume $l_{1}=1$,
{\small 
$$S=C(\lambda_{1}): y^{k}=(x^{3}-1) (x^{3}-\lambda_{1}^{3})^{k-1}, \quad \lambda_{1} \in \Omega.$$
}
By looking at the above equation for $S$, we may assume that
$$\alpha(x,y)=(\omega_{3} x,y),$$
so $\alpha^{3}=1$ and $\tau \alpha=\alpha \tau$.
The form of $\beta$ is as follows
{\small
$$\beta(x,y)=\left(\frac{\lambda_{1}}{x}, \frac{-\lambda_{1}^{3(k-1)/k}}{x^{3}(x^{3}-\lambda_{1}^{3})^{k-2}} y^{k-1}\right).$$
}

So, $\beta^{2}=(\alpha\beta)^{2}=1$ and $\beta \tau \beta^{-1}=\tau^{-1}$. Then,  
$\langle \alpha, \beta \rangle \cong {\mathbb D}_{3},$
and 
{\small
$$G=\langle \tau,\alpha,\beta: \tau^{k}=\alpha^{3}=\beta^{2}=(\alpha\beta)^{2}=1, \alpha \tau \alpha^{-1}=\tau, \beta \tau \beta=\tau^{-1}
\rangle =\langle \tau \rangle \rtimes \langle \alpha,\beta\rangle \cong {\mathbb Z}_{k} \rtimes {\mathbb D}_{3}.
$$
}

For $\lambda_{1}^{3} \neq 1$ it will happen that ${\rm Aut}(S)=G=N$ and $\bar{N}=\langle a,b\rangle \cong {\mathbb D}_{3}$. 
For $\lambda_{1}^{3}=1$ we obtain the previous case \ref{caso2}.

\begin{rema}
If $\lambda, \mu \in \Omega$, then $C(\lambda) \cong C(\mu)$ if and only if $\mu^{3} \in \{\lambda^{3}, \lambda^{-3}\}$, i.e., if and only if $\lambda$ and $\mu$ belong to the same orbit under the group 
$\langle U(\lambda)=\omega_{3} \lambda, V(\lambda)=\lambda^{-1}\rangle \cong {\mathbb D}_{3}$ of conformal automorphisms of $\Omega$. This, in particular, permits to observe that the field of moduli of $C(\lambda)$ is ${\mathbb Q}(\lambda^{3}+\lambda^{-3})$. For $\lambda^{3}=1$, $C(\lambda)$ is defined over its field of moduli ${\mathbb Q}$. If $\lambda^{3} \neq 1$, then our model $C(\lambda)$ 
is defined over a quadratic extension of its field of moduli. But, as a consequence of the results in \cite{AQ}, this surface is definable over its field of moduli. 
\end{rema}


\subsection{} Assume $\bar{N}$ contains ${\mathbb Z}_{6}$.
In this case, we may assume ${\mathbb Z}_{6}=\langle a(x)=\omega_{6} x\rangle$ and $t_{1}=1$, $t_{2}=\omega_{6}$, $t_{3}=\omega_{6}^{2}$, $t_{4}=\omega_{6}^{3}$, $t_{5}=\omega_{6}^{4}$ and $t_{6}=\omega_{6}^{5}$. Let $\alpha \in N$ be such that $\theta(\alpha)=a$ and $u=u_{\alpha}$ (so $u^{6} \equiv 1 \mod(k)$). In this case, 
{\small
$$S: y^{k}=(x-1)(x-\omega_{6})^{u}(x-\omega_{6}^{2})^{u^{2}}(x-\omega_{6}^{3})^{u^{3}}(x-\omega_{6}^{4})^{u^{4}}(x-\omega_{6}^{5})^{u^{5}},$$
}
where $1+u+u^{2}+u^{3}+u^{4}+u^{5} \equiv 0 \mod(k)$, and
{\small
$$\alpha(x,y)=\left(\omega_{6}x, \frac{ \omega_{6}^{\frac{1+u+u^{2}+u^{3}+u^{4}+u^{5}}{k}}  y^{u}}{(x-\omega_{6}^{5})^{\frac{u^{6}-1}{k}}}\right).$$
}

In this case, $S/N$ has signature $(0;6,6,k)$.

\subsection{} Assume $\bar{N}$ contains ${\mathbb Z}_{5}$.
In this case, we may assume ${\mathbb Z}_{5}=\langle a(x)=\omega_{5} x\rangle$ and $t_{1}=0$, $t_{2}=1$, $t_{3}=\omega_{5}$, $t_{4}=\omega_{5}^{2}$, $t_{5}=\omega_{5}^{3}$ and $t_{6}=\omega_{5}^{4}$. Let $\alpha \in N$ be such that $\theta(\alpha)=a$. Since $a$ fixes a point in ${\mathcal B}$, $u_{\alpha}=1$. In this case, 
{\small
$$S: y^{k}=x(x^{5}-1)^{l},$$
}
where $l \in \{1,\ldots,k-1\}$, $\gcd(k,l)=1$, $1+5l \equiv 0 \mod(k)$, 
{\small
$\alpha(x,y)=\left(\omega_{5}x, \omega_{5}^{\frac{1}{k}} y\right),$} and $S/N$ has signature $(0;5,k,5k)$.

\subsection{} Assume $\bar{N}$ contains ${\mathbb Z}_{4}$.
In this case, we may assume ${\mathbb Z}_{4}=\langle a(x)=i x\rangle$ and $t_{1}=0$, $t_{2}=1$, $t_{3}=i$, $t_{4}=-1$, $t_{5}=-i$ and $t_{6}=\infty$. Let $\alpha \in N$ be such that $\theta(\alpha)=a$. Since $a$ fixes a point in ${\mathcal B}$, $u_{\alpha}=1$. In this case, 
{\small
$$S: y^{k}=x(x^{4}-1)^{l},$$
}
where $l \in \{1,\ldots,k-1\}$, $\gcd(k,l)=1$, and $1+4l \nequiv 0 \mod(k)$, 
{\small
$\alpha(x,y)=\left(ix, \omega_{4k} y\right),$
} and $S/N$ has signature $(0;k,4k,4k)$.

\begin{rema}
Note that, for the case $1+4l \equiv k-1 \mod(k)$, the above curve has ${\mathbb D}_{4}$ contained  in $\bar{N}$.
\end{rema}

\subsection{} Assume $\bar{N}$ contains ${\mathbb Z}_{3}$.
In this case, we may assume ${\mathbb Z}_{3}=\langle a(x)=\omega_{3} x\rangle$ and $t_{1}=1$, $t_{2}=\omega_{3}$, $t_{3}=\omega_{3}^{2}$, $t_{4}=\lambda$, $t_{5}=\omega_{3}\lambda$ and $t_{6}=\omega_{3}^{2}\lambda$, where $\lambda \neq 0$ and $\lambda^{3} \neq 1$. 
Let $\alpha \in N$ be such that $\theta(\alpha)=a$ and $u=u_{\alpha}$ (so $u^{3} \equiv 1 \mod(k)$). In this case, 
{\small
$$S: y^{k}=(x-1)(x-\omega_{3})^{u}(x-\omega_{3}^{2})^{u^{2}}(x-\lambda)^{l}(x-\omega_{3}\lambda)^{ul}(x-\omega_{3}^{2}\lambda)^{u^{2}l},$$
}
where $l \in \{1,\ldots,k-1\}$, $\gcd(k,l)=1$, $(1+u+u^{2})(1+l) \equiv 0 \mod(k)$,
{\small
$$\alpha(x,y)=\left(\omega_{3}x, \frac{\omega_{3}^{\frac{(1+u+u^{2})(1+l)}{k}} y^{u}}{(x-\omega_{3}^{2})^{\frac{u^{3}-1}{k}}  (x-\omega_{3}^{2}\lambda)^{l\frac{u^{3}-1}{k}}   }\right),$$
}
and $S/N$ has signature $(0;3,3,k,k)$.

\subsection{} Assume $\bar{N}$ contains ${\mathbb Z}_{2}$. In this case, we may assume ${\mathbb Z}_{2}=\langle a(x)=-x\rangle$.
In this case, there are two situations.

\subsubsection{\rm Case 1}
If $a$ fixes two of the elements of ${\mathcal B}$. We may assume 
$t_{1}=0$, $t_{2}=1$, $t_{3}=-1$, $t_{4}=\lambda$, $t_{5}=-\lambda$ and $t_{6}=\infty$, where $\lambda^{2} \neq 0,1$. Let $\alpha \in N$ be such that $\theta(\alpha)=a$. Since $a$ fixes a point in ${\mathcal B}$, $u_{\alpha}=1$. In this case, 
{\small
$$S: y^{k}=x(x^{2}-1)^{l}(x^{2}-\lambda^{2})^{t},$$
}
where $l,t \in \{1,\ldots,k-1\}$, $\gcd(k,l)=\gcd(k,t)=1$, $2(1+l+t) \equiv 0 \mod(k)$,
{\small
$$\alpha(x,y)=(-x,\omega_{2k}y),$$
}
and $S/N$ has signature $(0;k,k,2k,2k)$.

\subsubsection{\rm Case 2}
In this case, $a$ does not have fixed points in ${\mathcal B}$. We may assume 
$t_{1}=1$, $t_{2}=-1$, $t_{3}=\lambda_{1}$, $t_{4}=-\lambda_{1}$, $t_{5}=\lambda_{2}$ and $t_{6}=-\lambda_{2}$, where $\lambda_{j}^{2} \neq 1,0$, $\lambda_{1}^{2} \neq \lambda_{2}^{2}$.
Let $\alpha \in N$ be such that $\theta(\alpha)=a$ and $u=u_{\alpha}$, so $u^{2} \equiv 1 \mod(k)$. In this case, 
{\small
$$S: y^{k}=(x-1)(x+1)^{u}(x-\lambda_{1})^{l}(x+\lambda_{1})^{ul}(x-\lambda_{2})^{t}(x+\lambda_{2})^{ut},$$
}
where $l,t \in \{1,\ldots,k-1\}$, $\gcd(k,l)=\gcd(k,t)=1$, $(1+u)(1+l+t) \equiv 0 \mod(k)$, 
{\small
$$\alpha(x,y)=\left(-x, \frac{\omega_{2k}^{(1+u)(1+l+t)} y^{u}}{ \left((x+1)(x+\lambda_{1})^{l}(x+\lambda_{2})^{t} \right)^{\frac{u^{2}-1}{k}}   }\right),$$
}
and $S/N$ has signature $(0;2,2,k,k,k)$.

\begin{rema}
For instance, if we take $k=6$, then $u \in \{1,5\}$.
\begin{enumerate}
\item If $u=1$, then
{\small $S: y^{6}=(x^{2}-1)(x^{2}-\lambda_{1}^{2})(x^{2}-\lambda_{2}^{2}), \; \alpha(x,y)=(-x,-y).$}

\item If $u=5$, then
{\small $S: y^{6}=(x-1)(x+1)^{5}(x-\lambda_{1})^{l}(x+\lambda_{1})^{5l}(x-\lambda_{2})^{t}(x+\lambda_{2})^{5t},$
$\alpha(x,y)=\left(-x,\frac{(-1)^{1+l+t}y^{5}}{ \left((x+1)(x+\lambda_{1})^{l}(x+\lambda_{2})^{t} \right)^{4}} \right).$}
If we take $l=t=1$ and $\lambda_{1}\lambda_{2}=-1$, we obtain the extra automorphism 
{\small $\beta(x,y)=\left(\frac{1}{x},\frac{y}{x^{3}}  \right).$}
\end{enumerate}
\end{rema}

\subsection{Example: The prime situation for $n=5$}
Let $p$ be a prime integer and $S$ be a closed Riemann surface of genus $g=2(p-1)$ such that $|{\rm Aut}(S)|=\mu p$, for some integer $\mu \geq 1$. 
In this case, $S$ admits a $p$-gonal automorphism $\tau$ such that $S/\langle \tau \rangle$ has signature $(0;p,p,p,p,p,p)$. 
Let $N \leq {\rm Aut}(S)$ be the normalizer of $\langle \tau \rangle$ and $\bar{N}=N/\langle \tau \rangle \leq {\rm PSL}_{2}({\mathbb C})$.
As observed above, the only possibilities for $\bar{N}$, apart from the trivial, are ${\mathbb Z}_{2}$, ${\mathbb Z}_{3}$, ${\mathbb Z}_{4}$, ${\mathbb Z}_{5}$, ${\mathbb Z}_{6}$, ${\mathbb Z}_{2}^{2}$, ${\mathbb D}_{3}$, ${\mathbb D}_{4}$, ${\mathbb D}_{6}$, ${\mathcal A}_{4}$, and ${\mathfrak S}_{4}$ (only for $p=3$).

\begin{rema}
If $p=2$ or $p \geq 23$, then $N={\rm Aut}(S)$. Moreover, if $p \geq 5$, then $$|N| \in \{2p,3p,5p,6p,4p,6p,8p,12p\}.$$
\end{rema}

Below, following the previous general computations, we write down the corresponding equations for $S$ in these cases for $p \geq 3$. 
\begin{enumerate}[leftmargin=15pt]
\item If $\bar{N} \cong {\mathfrak S}_{4}$, then {\small $S: y^{3}=x(x^{4}-1)$}, $\tau$ is central in $N$, and $S/N$ has signature {\small $(0;2,3,12)$}.

\item If $\bar{N} \cong {\mathcal A}_{4}$, then {\small $S: y^{p}=x(x^{2}-1)^{u}(x^{2}+1)^{u^{2}}$}, where {\small $u \in \{1,\ldots,p-1\}$} satisfies {\small $1+u+u^{2} \equiv 0 \mod(p)$}. In this case, $\tau$ is central only for {\small $p=3$} (for which {\small $u=1$}). The signature of $S/N$ is {\small $(0;3,3,2p)$}.

\item If $\bar{N} \cong {\mathbb D}_{6}$, then {\small $S: y^{p}=(x^{3}-1)(x^{3}+1)^{u}$}, where {\small $1 \leq u \leq p-1$} satisfies {\small $3(1+u) \equiv 0 \mod(p)$}. In particular, if {\small $p \neq 3$}, then {\small $u=p-1$}. In this case, $\tau$ is central in $N$ only for {\small $p=3$} and {\small $u=1$}. The signature of $S/N$ is {\small $(0;2,6,2p)$}.

\item If $\bar{N} \cong {\mathbb D}_{4}$, then {\small $S: y^{p}=x(x^{4}-1)^{\frac{p-1}{2}}$}. In this case, $\tau$ is not central in $N$ and the signature of $S/N$ is {\small $(0;2,2p,4p)$}.

\item If $\bar{N} \cong {\mathbb Z}_{2}^{2}$, then we have two cases: 
\begin{enumerate}
\item {\small $S: y^{p}=x(x^{2}-1)^{l}(x^{2}+1)^{t}$}, where {\small $l,t \in \{1,\ldots,p-1\}$} satisfy that {\small $1+l+t \equiv 0 \mod(p)$}, in which case $\tau$ is central and the signature of $S/N$ is {\small $(0;2p,2p,2p)$}; or

\item {\small $S: y^{p}=x(x^{2}-\lambda^{2})^{\frac{p-1}{2}}(x^{2}-\lambda^{-2})^{\frac{p-1}{2}}$}, where {\small $\lambda^{2} \neq 0,1$}. In this case, again $\tau$ is central and the signature of $S/N$ is {\small $(0;2,2,p,2p)$}.

\end{enumerate}

\item If $\bar{N} \cong {\mathbb D}_{3}$, then {\small $S: y^{p}=(x^{3}-1)(x^{3}-\lambda^{2})^{p-1}$}, where {\small $\lambda^{3} \neq 0,1$}. In this case, $\tau$ is not central in $N$ and the signature of $S/N$ is {\small $(0;2,2,3p)$}. 

\item If $\bar{N} \cong {\mathbb Z}_{6}$, then {\small $S: y^{p}=(x-1)(x-\omega_{6})^{u}(x-\omega_{6}^{2})^{u^{2}}(x-\omega_{6}^{3})^{u^{3}}(x-\omega_{6}^{4})^{u^{4}}(x-\omega_{6}^{5})^{u^{5}}$},
where {\small $u \in \{1,\ldots,p-1\}$} satisfies that {\small $1+u+u^{2}+u^{3}+u^{4}+u^{5} \equiv 0 \mod(p)$}. In this case, $\tau$ is central in $N$ only for {\small $p=3$} (for which {\small $u=1$}). The signature of $S/N$ is {\small $(0;6,6,p)$}.

\item If $\bar{N} \cong {\mathbb Z}_{5}$, then {\small $S: y^{p}=x(x^{5}-1)^{l}$}, where {\small $l \in \{1,\ldots,p-1\}$} satisfies that {\small $1+5l \equiv 0 \mod(p)$}. In this case, $\tau$ is central in $N$ and the signature of $S/N$ is {\small $(0;5,p,5p)$}.

\item If $\bar{N} \cong {\mathbb Z}_{4}$, then {\small $S: y^{p}=x(x^{4}-1)^{l},$} where {\small $l \in \{1,\ldots,p-1\}$, $\gcd(k,l)=1$}, and {\small $1+4l \nequiv 0 \mod(p)$}. In this case, $\tau$ is central in $N$ and the signature of $S/N$ is {\small $(0;p,4p,4p)$}.

\item If $\bar{N} \cong {\mathbb Z}_{3}$, then {\small $S: y^{p}=(x-1)(x-\omega_{3})^{u}(x-\omega_{3}^{2})^{u^{2}}(x-\lambda)^{l}(x-\omega_{3}\lambda)^{ul}(x-\omega_{3}^{2}\lambda)^{u^{2}l}$},
where {\small $\lambda^{3} \neq 0,1$} and {\small $u,l \in \{1,\ldots,p-1\}$} satisfy that either {\small $l=p-1$} or {\small $1+u+u^{2} \equiv 0 \mod(p)$}. In this case, $\tau$ is central in $N$ only for {\small $u=1$} and {\small $3(1+l) \equiv 0 \mod(p)$}. In particular, for {\small $p \geq 5$}, $\tau$ is central in $N$ only for {\small $l=p-1$} and {\small $u=1$}. The signature of $S/N$ is {\small $(0;3,3,p,p)$}.

\item If $\bar{N} \cong {\mathbb Z}_{2}$, then we have two cases: 
\begin{enumerate}
\item {\small $S: y^{p}=x(x^{2}-1)^{l}(x^{2}-\lambda^{2})^{t}$}, where {\small $\lambda^{2} \neq 0,1$}, and {\small $l,t \in \{1,\ldots,p-1\}$} satisfy that {\small $1+l+t \equiv 0 \mod(p)$}. In this case $\tau$ is central in $N$ and the signature of $S/N$ is {\small $(0;p,p,2p,2p)$}; or

\item {\small $S: y^{p}=(x-1)(x+1)^{u}(x-\lambda_{1})^{l}(x+\lambda_{1})^{ul}(x-\lambda_{2})^{t}(x+\lambda_{2})^{ut}$},  where $\lambda_{j}^{2} \neq 1,0$, $\lambda_{1}^{2} \neq \lambda_{2}^{2}$
and either {\small $u=p-1$} or {\small $1+l+t \equiv 0 \mod(p)$}. In this case, $\tau$ is central in $N$ only for {\small $u=1$} and {\small $1+l+t \equiv 0 \mod(p)$}. The signature of $S/N$ is {\small $(0;2,2,p,p,p)$}. 
\end{enumerate}
\end{enumerate}

\begin{rema} 
If either (a) $p \geq 7$, or (b) $p=5$ and $\bar{N} \neq {\mathbb Z}_{5}$, or (c) $p=3$ and $\bar{N} \neq {\mathbb Z}_{3}, {\mathbb Z}_{6}, {\mathbb D}_{3}, {\mathbb D}_{6}, {\mathcal A}_{4}, {\mathfrak S}_{4}$,
then (by the Schur-Zassenhaus theorem), $N \cong {\mathbb Z}_{p} \rtimes \bar{N}$.
Moreover, as for $\bar{N} \in \{{\mathfrak S}_{4}, {\mathbb D}_{4}, {\mathbb Z}_{2}^{2}, {\mathbb Z}_{5}, {\mathbb Z}_{4}\}$, $\tau$ is central in $N$, then $N \cong {\mathbb Z}_{p} \times \bar{N}$.
\end{rema}




\begin{thebibliography}{99}


\bibitem{Accola}
Accola, R.: 
Topics in the Theory of Riemann Surfaces. 
Springer-Verlag (1994)


\bibitem{AK}
Antoniadis, J. A. and Kontogeorgis, A.:
On cyclic covers of the projective line.
{\it Manuscripta Math.} {\bf 121}, 105--130 (2006) 



\bibitem{AQ}
Artebani, M. and Quispe, S. 
Fields of moduli and fields of definition of odd signature curves. 
{\it Arch. Math. (Basel)} {\bf 99} (4), 333--344 (2012)

 
\bibitem{BCI}
Bartolini, G., Costa, A.F.  and Izquierdo, M.:
On automorphisms groups of cyclic $p$-gonal Riemann surfaces.
{\it Journal of Symbolic Computation} {\bf  57},  61--69 (2013)
 
 
\bibitem{Brandt1}
Brandt, R.: 
\"Uber die Automorphismengruppen von Algebraischen Funktionenkörpern, Ph.D. thesis. Essen Universit\"at, 1988.

\bibitem{Brandt2}
Brandt, R., and Stichtenoth, H.: 
Die Automorphismengruppen hyperelliptischer Kurven. 
{\it Manuscripta Math.} {\bf 55} (1), 83--92 (1986).  
 
 

\bibitem{CR}
Carocca, A. and Reyes-Carocca, S.:
The locus of Riemann surfaces of genus $2(p-1)$ with $4p$ automorphisms, Preprint
\url{https://arxiv.org/pdf/2504.02771} (2025).



\bibitem{Greenberg}
Greenberg, L.:
Maximal groups and signatures. 
Discontinuous groups and Riemann surfaces (Proc. Conf., Univ. Maryland, College Park, Md., 1973), 
Ann. of Math. Studies, No. {\bf 79},
Princeton University Press, 1974, pp. 207--226. 


\bibitem{Hidalgo:pgrupo}
Hidalgo, R.A.:
$p$-Groups acting on Riemann surfaces.
{\it Journal of Pure and Applied Algebra} {\bf 222}, 4173--4188 (2018)


\bibitem{HQS}
Hidalgo, R.A., Quispe, S. and Shaska, T.:
Generalized Superelliptic Riemann Surfaces.
{\it Transformation Groups} 
\url{https://doi.org/10.1007/s00031-025-09911-5}

\bibitem{HQS2}
Hidalgo, R.A., Quispe, S. and Shaska, T.:
Generalized Superelliptic Riemann Surfaces.
https://arxiv.org/abs/1609.09576

\bibitem{Hurwitz}
Hurwitz, A.:
\"Uber algebraische Gebilde mit eindeutigen Transformationen in sich.
Math. Ann. {\bf 41}, 403--442 (1893)

\bibitem{IY}
Ishii, N. and Yoshida, K.:
The automorphism group of a cyclic $p$-gonal curve.
{\it Tsukuba J. Math.} {\bf 31}, 1--37 (2007).


\bibitem{K}
Kontogeorgis, A.:
The group of automorphisms of cyclic extensions of rational function fields.
{\it J. Algebra} {\bf 216}, 665--706 (1999)


\bibitem{MPRZ}
Mu\c{c}o, R., Pjero, N., Ruci, E., and Zhupa, E.:
Classifying families of superelliptic curves.
{\it Albanian J. of Math.} {\bf 8} (2014), 23--35.

\bibitem{Sanjeewa}
Sanjeewa, R.:
Automorphism groups of cyclic curves defined over finite fields of any characteristic.
{\it Albanian J. Math.} {\bf 3}, 131--160 (2009)

\bibitem{SS}
Sanjeewa, R. and Shaska, T.:
Determining equations of families of cyclic curves.
{\it Albanian J. Math.} {\bf 2}, 199--213 (2008)


\bibitem{Schwarz}
Schwartz, H.A.:
\"Uber diejenigen algebraischen Gleichungen zwischen zwei ver\"anderlichen Gr\"o{\ss}en, welche eine schaar rationaler, eindeutig umkehrbarer 
Transformationen in sich selbst zulassen.
Journal f\"ur die reine und angewandte Mathematik {\bf 87}, 139--145 (1890) 


\bibitem{Singerman}
Singerman, D.:
Finitely maximal Fuchsian groups.
{\it J. London Math. Soc.} {\bf 6} 29--38 (1972)


\bibitem{Wootton1}
Wootton, A.:
Defining equations for cyclic covers of the Riemann sphere.
{\it Israel Journal of Mathematics} {\bf 157}, 103--122 (2007)

\end{thebibliography}
\end{document}